\documentclass[a4paper,11pt]{article}
\usepackage{graphicx} % Required for inserting images

\usepackage[margin=3cm]{geometry}
\usepackage{tikz}

\usepackage{float}
\usepackage{amsmath,amsthm,amssymb}
\usepackage{hyperref}
\usepackage{physics}
\usepackage{comment}
\usepackage{enumerate}

\newtheorem{theorem}{Theorem}[section]
\newtheorem{corollary}[theorem]{Corollary}
\newtheorem{proposition}[theorem]{Proposition}
\newtheorem{lemma}[theorem]{Lemma}

\theoremstyle{definition}
\newtheorem{definition}[theorem]{Definition}
\newtheorem{example}[theorem]{Example}
\newtheorem{remark}[theorem]{Remark}

\newtheorem{observation}[theorem]{Observation}

\def\R{\mathbb R}

\DeclareMathOperator{\vol}{vol}
\newcommand{\Prob}[1]{\mathbb{P}\left(#1\right)}

\newcommand{\Exp}[1]{\mathbb{E}\left[#1\right]}
\newcommand{\cartesianproduct}{\,\square\,}

\usepackage[dvipsnames]{xcolor}

\definecolor{DarkDandelion}{HTML}{D97400}

\title{The edge-isoperimetric number of graphs and their powers}

\author{Aida Abiad\thanks{\texttt{a.abiad.monge@tue.nl}\\
Department of Mathematics and Computer Science, Eindhoven University of Technology, Eindhoven, The Netherlands\\
Department of Mathematics and Data Science, Vrije Universiteit Brussel, Belgium}
\quad
Nils van de Berg\thanks{\texttt{n.p.v.d.berg@tue.nl}\\
Department of Mathematics and Computer Science, Eindhoven University of Technology, Eindhoven, The Netherlands}
\quad
Emanuel Juliano\thanks{\texttt{emanuelsilva@dcc.ufmg.br}\\
Department of Computer Science, Federal University of Minas Gerais, Brazil}
\quad
Harper Reijnders\thanks{\texttt{l.e.r.m.reijnders@tue.nl}\\
Department of Mathematics and Computer Science, Eindhoven University of Technology, Eindhoven, The Netherlands}
\\
Robin Simoens\thanks{\texttt{Robin.Simoens@UGent.be}\\
Department of Mathematics: Analysis, Logic and Discrete Mathematics, Ghent University, Belgium\\
Department of Mathematics, Universitat Politècnica de Catalunya, Spain}
\quad
Thijs van Veluw\thanks{\texttt{Thijs.vanVeluw@ugent.be}\\
Department of Mathematics, Computer Science and Statistics, Ghent University, Belgium\\
Department of Mathematics and Computer Science, Eindhoven University of Technology, Eindhoven, The Netherlands}
\quad
Jim Wittebol\thanks{\texttt{Jim.Wittebol@vub.be}\\
Department of Mathematics and Data Science, Vrije Universiteit Brussel, Belgium}}

\date{}

\begin{document}

\maketitle

\begin{abstract}
We obtain several sharp spectral bounds, approximations, and exact values for the isoperimetric number and related edge‑expansion parameters of graphs. Our results focus on graph powers and on families of graphs with rich algebraic or geometric structure, including distance‑regular graphs and graphs arising from finite geometries, among others. Our proofs use techniques from spectral graph theory, linear optimization, finite geometry, and probability, yielding new machinery for analysing edge-expansion phenomena in highly structured graphs.\\

\noindent   \textbf{MCS:} 05C35, 05C85, 05C50, 05E30
\end{abstract}

%%%%%%%%%%%%%%%%%%%%%%%%%%%%%%%%%%%%%%%%%%%%%%%%%%%%%%%%%%%%%%%%%%%%%%%%%%%
\section{Introduction}
%%%%%%%%%%%%%%%%%%%%%%%%%%%%%%%%%%%%%%%%%%%%%%%%%%%%%%%%%%%%%%%%%%%%%%%%%%

Edge-expansion is a fundamental property in graphs that is captured by a variety of different but related parameters, for example the isoperimetric number, the Cheeger constant, and the sparsity number. 
%Formally, 
Consider a graph $G=(V, E)$ and a subset of its vertices $S \subseteq V$. The \emph{volume} of $S$, $\vol(S)$, is defined as the sum of the degrees of the vertices in $S$. Let $\partial_G(S)$ denote the set of edges with exactly one endpoint in $S$. We define the following expansion parameters of $S$
\begin{equation*}
    i_G(S) = \frac{|\partial_G(S)|}{|S|},\quad h_G(S) = \frac{|\partial_G(S)|}{\vol(S)}, \quad \sigma_G(S) = \frac{|\partial_G(S)|}{|S||V\setminus S|}. 
\end{equation*}
When the graph \(G\) is clear from the context, we use \(\partial S\), \(i(S)\), \(h(S)\) and \(\sigma(S)\) instead of \(\partial_G(S)\), \(i_G(S)\), \(h_G(S)\) and \(\sigma_G(S)\). Let $n = |V|$ be the order of the graph. The expansion parameters for a graph are defined as the minimum over all appropriate sets. In particular,
\begin{itemize}
    \item The \emph{isoperimetric number} (or \emph{edge-isoperimetric number}) of $G$ is
\begin{equation*}
    i(G) = \min\left\{i_G(S)\colon\,S \subseteq V,\,1 \leq |S| \leq\frac12 n\right\}.
\end{equation*}
\item The \emph{Cheeger constant} of $G$ is   
\begin{equation*}
    h(G) = \min\left\{h_G(S)\colon\,S \subseteq V,\,1 \leq \vol(S) \leq \vol(V)/2\right\}.
\end{equation*}

\item The \emph{sparsity} of $G$ is   
\begin{equation*}
    \sigma(G) = \min\left\{\sigma_G(S)\colon\,S \subseteq V,\,\emptyset\neq S\neq V\right\}.
\end{equation*}
\end{itemize}

\noindent These parameters are highly related, especially in the case of regular graphs. In this paper, we focus mostly on the isoperimetric number. Note that there are also vertex-expansion parameters such as the \emph{vertex-isoperimetric number}
\begin{equation*}
    i_V(G) = \min\left\{\frac{|N_G(S)|}{|S|}\colon\,S \subseteq V,\,1 \leq |S| \leq\frac12 n\right\}.
\end{equation*}
where \(N_G(S)\) is the set of vertices in \(V(G)\setminus S\) that are adjacent to at least one vertex in \(S\).
%(see for example \cite[Section~4.6]{HN2006})
However, vertex-expansion parameters are beyond the scope of this paper.

The study of edge-expansion parameters has deep connections to spectral graph theory, combinatorial optimization, and theoretical computer science. Since computing the isoperimetric number is NP-hard \cite[Theorem 2]{kaibel2004expansion}, a lot of literature focuses on bounding this parameter for specific graph classes or approximating it via spectral methods. In particular, isoperimetric-type inequalities relate $i(G)$ to the second-largest eigenvalue of the adjacency matrix or the second-smallest eigenvalue of the Laplacian matrix%\emanuel{second-smallest for Laplacian.}
, providing a bridge between combinatorial and spectral properties. These inequalities have been instrumental in the analysis of expander graphs, random walks, and mixing times of Markov chains (see \cite{alon1985lambda1, Chung1997, JS1989, LK1999}). Two celebrated spectral bounds %(see \cite{alon1985lambda1,M1989}) 
are 
\begin{equation}
\label{eq:Mohar}
\frac{\mu_2}{2}\leq i(G) \leq \sqrt{\mu_2(2d_1-\mu_2)},
\end{equation}

where $\mu_2$ is the second smallest Laplacian eigenvalue and $d_1$ is the largest degree \cite{alon1985lambda1,M1989}. These two bounds have been a staple of spectral graph theory. In recent years other spectral bounds have followed, such as the following upper bound for the isoperimetric number of regular graphs (see for example \cite[Lemma 5]{QKM2020}):
\begin{equation}
    \label{eq:QKMmun}
    i(G) \le \frac{\lceil \frac n2 \rceil}{n}\mu_n,
\end{equation}
where $\mu_n$ is the largest Laplacian eigenvalue. Also, an alternative upper bound  in terms of the graph degree sequence and the Laplacian eigenvalues is shown in \cite{AFHP2014}, and
an improvement of the lower bound \eqref{eq:Mohar} via a linear program (LP) that uses information of the Laplacian eigenvectors in \cite{DM2019}.

In this paper, we derive several sharp spectral bounds, approximations and exact values for the isoperimetric number and related edge-expansion parameters, focusing on graph powers and graphs with a rich algebraic and geometric nature (such as distance-regular graphs and graphs coming from finite geometries). Our proofs use techniques from spectral graph theory, linear optimization, finite geometry and probability. 

We start in Section \ref{sec:powers} by deriving sharp bounds on the isoperimetric number of graph powers, contributing to the existing literature in this direction  \cite{cartesianpowersregulargrpahs,hypercubeedgeisopowers,V2026}. In particular, we show that the lower bound \eqref{eq:Mohar} holds for the more general class of graph  powers, even when the spectrum of the graph power and the base graph are not related, and we also derive an alternative upper bound. Both bounds are shown using the polynomial interlacing method. They depend on the spectrum of the base graph and on a polynomial which can be chosen freely; thus, one needs to make a good choice of the polynomial to obtain the strongest bound. We propose linear optimization methods to compute the new polynomial eigenvalue bounds, derive the best possible bounds for small values of the power, and present multiple graphs that exhibit tightness for the new bounds.

An alternative lower bound for the isoperimetric number comes from a relaxation of the LP for the sparsity parameter \cite{LLR1995}. This LP has been studied for strongly regular graphs in \cite{S2005}. In Section \ref{sec:DRG}, we derive a closed formula for the optimal value of this LP for distance-regular graphs. We show that the optimal value of the LP can be computed using only the intersection array, without having to construct the distance-regular graph. This contributes to the previous literature on expansion parameters of distance-regular graphs  \cite{KL2013,QKM2020,S2005}. 

Although computing the isoperimetric number of a general graph is NP‑complete, its exact value has been determined for several notable graph families \cite{Aslan2013,M1989}. 
In Section \ref{sec:finitegeometry}, we propose a new method for establishing exact isoperimetric values for graphs arising from finite geometry. Our approach is based on a simple but effective observation linking the isoperimetric number to the existence of so-called tight sets. This connection yields a novel and easy method for exactly determining isoperimetric numbers in graphs originating from finite geometries, while previously only bounds have been studied, see for example \cite{levi}. We demonstrate its power and applicability by obtaining explicit formulas for the isoperimetric number of certain Grassmann graphs and various polar graphs.

Despite the existence of eigenvalue bounds for several NP-hard graph parameters, most of these seem not to be determined by the graph spectrum \cite{MRC,bch2015,h2020,h1996,lwyl2010}. In Section \ref{sec:NDS} we demonstrate that the isoperimetric number is one of them. We do so by providing an infinite family of pairs of cospectral connected regular graphs with different isoperimetric numbers. 

Motivated by similar research for other random graph models~\cite{Bollobas1988, Alon1986,BHKL2008}, in Section \ref{sec:split}, we establish the isoperimetric number of the random split graph model and show that with high probability the spectral upper bound from \cite{AFHP2014} is tight for that model.

%%%%%%%%%%%%%%%%%%%%%%%%%%%%%%%%%%%%%%%%%%%%%%%%%%%%%%%%%%%%%%%%%%%%%%%%%%%
\section{Preliminaries}
%%%%%%%%%%%%%%%%%%%%%%%%%%%%%%%%%%%%%%%%%%%%%%%%%%%%%%%%%%%%%%%%%%%%%%%%%%%

For a graph $G$, we denote the number of its vertices by $n$, the degrees of the $n$ vertices by $d_1\ge \dots \ge d_n$, the adjacency eigenvalues by $\lambda_1\geq\dots \geq\lambda_n$, and the Laplacian eigenvalues by $0=\mu_1\leq\dots\leq\mu_n$.
If \(G\) is \(k\)-regular, that is, if \(d_1=\dots=d_n=k\), then \(\lambda_i+\mu_i=k\) and in particular \(\lambda_1=k\).
For vertices $u$ and $v$, we denote their distance by $d(u,v)$, and we denote the minimum degree of $G$ by $\delta(G)=d_n$.

The \emph{Cartesian product} of two graphs $G$ and $H$ is the graph $G \cartesianproduct H$ with vertex set $V(G) \times V(H)$ such that the vertices $(v,w)$ and $ (v',w')$ are adjacent whenever $v$ is adjacent to $v'$ in $G$ and $w = w'$, or $v = v'$ and $w$ is adjacent to $w'$ in $H$.

Eigenvalue interlacing is a powerful method for obtaining spectral bounds in graph theory. We briefly recall interlacing of quotient matrices. For more details, see \cite[Sections 2.3 and 2.5]{spectra}.

Let $M$ be a symmetric $n\times n$ matrix, and let $\mathcal P=\{P_1,\dots,P_m\} $ be a partition of the set $\{1,\dots, n\}$ into non-empty subsets. The \emph{quotient matrix} of $M$ with respect to the partition $\mathcal P$ is defined as the $m\times m$ matrix $B$ with
$$B_{ij}=\frac1{|P_i|} \sum_{a \in P_i}\sum_{b\in P_j} M_{ab}.$$
The value $B_{ij}$ can be interpreted as the average row sum of the submatrix of $M$ on the rows indexed by $P_i$ and the columns indexed by $P_j$. If $M$ is the adjacency matrix of a graph, then $B_{ij}$ counts the average number of neighbours in $P_j$ among vertices in $P_i$.

The partition $\mathcal P$ is called \emph{equitable} or \emph{regular} if for all $i$ and $j$, the submatrix of $M$ on $P_i$ and $P_j$ has the property that every row sums to $B_{ij}$. A subset $S$ of the vertices of a regular graph is called \emph{intriguing}\footnote{Intriguing sets are sometimes also called: regular sets, equitable bipartitions, or perfect 2-colourings.} if the partition $\{S,S^c\}$ (where $S^c$ is the complement of $S$ in $\{1,\dots,n\}$) of the vertex set is non-trivial and equitable with respect to the adjacency matrix. In that case, every vertex of $S$ has $B_{11}$ neighbours in $S$ and $B_{12}$ neighbours outside of $S$, and every vertex outside of $S$ has $B_{21}$ neighbours in $S$ and $B_{22}$ neighbours outside of $S$.

If $\alpha_1\ge \dots \ge \alpha_n$ and $\beta_1\ge \dots \ge \beta_m$ are the eigenvalues of $M$ and $B$, respectively, then the eigenvalues of $B$ \emph{interlace} those of $A$, meaning that $\alpha_{n-m+i} \le\beta_i \le  \alpha_i$ for all $i\in\{1,\dots,m\}$. If the interlacing is \emph{tight} (that is, there exists a $k\in\{0,\dots,m\}$ such that $\alpha_i=\beta_i$ for $1\le i\le k$ and $\alpha_{n-m+i}=\beta_i$ for $k+1\le i \le m$, where in the edge cases $k=0$ or $k=m$ one of these two conditions is vacuous), then the partition is equitable \cite{spectra}. In particular, if $S$ is a subset of vertices of a regular graph and the eigenvalues of the quotient matrix $B$ of $A$ with respect to the partition $\{S,S^c\}$ tightly interlace the eigenvalues of $A$, then $S$ is an intriguing set.

Tight interlacing is used for establishing the equality case in the following theorem.

\begin{theorem}[{\cite[Theorem~3.5]{Haemers1995}}]\label{thm:interlacing}
    Let \(G\) be a \(k\)-regular graph on \(n\) vertices with Laplacian eigenvalues $0=\mu_1\le\dots\le\mu_n$. For any subset \(S\) of the vertices we have
    \[\left(1-\frac{|S|}{n}\right)\mu_2\leq\frac{|\partial S|}{|S|}\leq\left(1-\frac{|S|}{n}\right)\mu_n.
    \]
    If equality holds in either inequality, then $S$ is an intriguing set.
\end{theorem}
\begin{comment}
\begin{proof}
    Apply Lemma~\ref{lemma:interlacing} with $M=A$, where \(A\) is the adjacency matrix, which has constant row sum $\lambda_1$ because the graph is regular. Because $i_A(S)=|\partial S|/|S|$, the result follows.
\end{proof}
\end{comment}
For regular graphs, the lower bound of \eqref{eq:Mohar} now follows from applying the lower bound of Theorem~\ref{thm:interlacing} to a set $S$ with $i(G)=i_G(S)$ because $1-\frac1{n}|S| \ge \frac 12$. Similarly, \eqref{eq:QKMmun} follows from applying the upper bound of Theorem~\ref{thm:interlacing} to any set $S$ of size $\lfloor \frac n2\rfloor$. The upper bound of \eqref{eq:Mohar} does not follow as straightforwardly as the other two bounds, requiring also a spectral partitioning argument \cite{M1989}.

Specializing further from regular graphs, we consider distance-regular graphs. A connected graph $G$ with diameter $D$ is called \emph{distance-regular} if there are constants $c_i, a_i, b_i$ such that for all $i\in\{0,\dots,D\}$ and all vertices $u$ and $v$ at distance $d(u, v)=i$, there are exactly $c_i$ neighbours of $v$ at distance $i-1$ from $u$, exactly $a_i$ neighbours of $v$ at distance $i$ from $u$, and exactly $b_i$ neighbours of $v$ at distance $i+1$ from $u$.

In a distance-regular graph, %$k_i$,
the number of vertices at distance exactly $i$ from a vertex $v$, is independent of the choice of $v$. The \emph{intersection numbers} $p_{ij}^h$ denote the number of vertices $w$ that are at distance $i$ from $u$ and $j$ from $v$ for a pair $u,v$ with $d(u,v) = h$, and are independent of the choice of vertices. We have $c_i=p^i_{i-1,1}$, $a_i=p^i_{i,1}$, and $b_i=p^i_{i+1,1}$.

%%%%%%%%%%%%%%%%%%%%%%%%%%%%%%%%%%%%%%%%%%%%%%%%%%%%%%%%%%%%%%%%%%%%%%%%%%%
\section{Eigenvalue bounds for the isoperimetric number of graph powers}\label{sec:powers}
%%%%%%%%%%%%%%%%%%%%%%%%%%%%%%%%%%%%%%%%%%%%%%%%%%%%%%%%%%%%%%%%%%%%%%%%%%%

For a positive integer $t$, the $t$-th \emph{graph power} of a graph $G =(V, E)$, denoted by $G^t$, is the graph with vertex set $V$ where two distinct vertices are adjacent if there is a path in $G$ of length at most $t$ between them.
Graph powers help design efficient algorithms for specific combinatorial optimization problems, for example the travelling salesman problem
\cite{AB1995,BC2000}. In distributed computing, the $t$-th power of a graph $G$ represents the possible flow of information during $t$ rounds of communication in a distributed network of processors organized according to $G$ \cite{l92}.
The investigation of the behaviour of various parameters of powers of a fixed graph leads to many interesting problems, some of which are motivated by questions in information theory, communication complexity, Euclidian geometry and Ramsey theory, see \cite{Apowers,powersthesis} for an overview. There are multiple results on combinatorial parameters of graph powers, such as the chromatic number \cite{AM2002,H2009,KP2016}, the independence number \cite{ADFZ2023,acf2019}, the Shannon capacity \cite{ADF2025,AL2007v1}, and the rainbow connection number \cite{Basavaraju2014RainbowProducts}, among others. Less is known about spectral properties of the graph powers in terms of the spectral properties of the base graph  \cite{acfnz21,Das2013LaplacianGraph,JZ2017}.

Motivated by the above, we investigate the isoperimetric number of graph powers, with a main focus on deriving sharp spectral bounds  (Theorem~\ref{theorem:lowerboundi(G^t)} and Theorem~\ref{theorem:upperboundi(G^t)}) using the eigenvalues of the base graph $G$ instead of the eigenvalues of the power graph. The graph $G^t$ is denser than $G$, so the isoperimetric number increases with $t$. Our lower bounds similarly increase with $t$. Our bounds are based on the polynomial interlacing method developed in \cite{acf2019}, which has been successfully used later, see e.g. \cite{AZ2025}. The derived bounds using this polynomial interlacing depend on a polynomial that can be chosen freely. For the case $t=2$ we find the polynomial that gives the strongest possible bound (Theorems~\ref{theorem:lowerboundi(G^2)closed} and \ref{theorem:upperboundi(G^2)closed}), and we provide linear programming methods to find the best bound for general $t$ (LP \eqref{eq:iso_LP_lower} and \eqref{eq:iso_LP_upper}).

Lower bounds on the expansion of $G^t$ were also investigated in \cite{isoperimetricgraphpowers}; a well-known operation to improve graph expansion is the $t$-th power of $G$, which has a natural correspondence to simulating the random walk on $G$ for $t$ steps. However this was done in a different setting, in particular for \emph{weighted graphs}: graphs (with self loops) together with a weight function $w:E(G) \to \R_{\ge 0}$ on the edges. Accordingly, there are weighted variants of the definitions of the adjacency matrix, graph power, and isoperimetric number. The \emph{weighted adjacency matrix} $W$ of a weighted graph has $W_{uv}=w(\{u,v\})$ whenever $u$ and $v$ are adjacent, and 0 otherwise. The \emph{weighted graph power} $G^t$ is the weighted graph with weighted adjacency matrix $W^t$. The \emph{weighted isoperimetric number} can be expressed as $$i_W(G)= \min\left\{\frac1{|S|} \sum_{u \in S, v \notin S} w(u,v)\colon\,S \subseteq V,\,1\le |S|\le\frac n2\right\}.$$

The authors of \cite{isoperimetricgraphpowers} showed that for $1$-regular weighted graphs (meaning that the weights of the edges out of any vertex sum to 1) with all diagonal weights at least $\frac12$ it holds that
$$i_{W^t}(G^t) \ge \frac{1}{20}(1 - (1 - i_W(G))^{\sqrt{t}}).$$
However, the weighted graph power is automatically spectrally related to the base weighted graph. In our unweighted setting, the spectra of a graph $G$ and its graph power $G^t$ are not related to each other in general (see for example \cite{acfnz21,CHEN2019418,Das2013LaplacianGraph,DG2016,JZ2017}).  

%%%%%%%%%%%%%%%%%%%%%%%%%%%%%%%%%%%%%%%%%%%%%%%%%%%%%%%%%%
\subsection{Bounds when the spectra of \texorpdfstring{$G$}{G} and \texorpdfstring{$G^t$}{Gt} are related}
%%%%%%%%%%%%%%%%%%%%%%%%%%%%%%%%%%%%%%%%%%%%%%%%%%%%%%%%%%

In the special case where the spectra of $G$ and $G^t$ are related by a polynomial (that is, if there exists a polynomial $p \in \mathbb{R}[x]$ such that $A(G^t)=p(A)$), we can use the existing bounds \eqref{eq:Mohar} and \eqref{eq:QKMmun} to obtain bounds on the isoperimetric number of $G^t$ in terms of the spectrum of $G$ as follows. In particular, we can do this for distance-regular graphs, and, more generally, for \emph{partially distance-polynomial graphs} (see \cite{partially distance-polynomial}).

For a graph $G$ on $n\ge 2$ vertices and a polynomial $p\in \R[x]$, we define the following parameters.
\begin{align*}
    \Lambda(p) &= \max_{2\le i \le n} p(\lambda_i),\\
    \lambda(p) &= \min_{2\le i \le n} p(\lambda_i).
\end{align*}

\begin{lemma}\label{lemma:power}
    Consider a regular graph $G$ on $n\ge 2$ vertices such that there is a polynomial $p \in \mathbb{R}[x]$ with $p(A) = A(G^t)$. The following inequalities hold:
    \begin{align*}
    i(G^t) &\geq \frac{1}{2}( p(\lambda_1) - \Lambda(p)),\\
    i(G^t) &\le \sqrt{p(\lambda_1)^2-\Lambda(p)^2},\\
    i(G^t) &\le \frac{\lceil \frac n2 \rceil}{n}(p(\lambda_1)-\lambda(p)).
    \end{align*}
\end{lemma}
\begin{proof}
    Since $G$ is regular, the all-ones vector is an eigenvector for the largest eigenvalue $\lambda_1$ of $G$. Since $p(A)=A(G^t)$, the all-ones vector is also an eigenvector of $G^t$, with eigenvalue $p(\lambda_1)$. Now $G^t$ is $p(\lambda_1)$-regular, and $p(\lambda_1)$ is the largest adjacency eigenvalue of $G^t$. The other adjacency eigenvalues of $G^t$ are $p(\lambda_i)$ for $i=2,\dots,n$. The result follows from the bounds \eqref{eq:Mohar} and \eqref{eq:QKMmun} applied to $G^t$.
\end{proof}

%%%%%%%%%%%%%%%%%%%%%%%%%%%%%%%%%%%%%%%%%%%%%%%%%%%%%%%%%%
\subsection{Bounds when the spectra of \texorpdfstring{$G$}{G} and \texorpdfstring{$G^t$}{Gt} are not related}\label{sec:unrelatedspectra}
%%%%%%%%%%%%%%%%%%%%%%%%%%%%%%%%%%%%%%%%%%%%%%%%%%%%%%%%%%

In the more usual case where the spectra of $G$ and $G^t$ are not related, we need a different approach. In \cite{acf2019}, a method using interlacing of polynomials of the adjacency matrix of a graph was developed to generalize the inertia and Hoffman ratio bounds on the independence number of a graph to bounds on the independence number of a graph power. This polynomial interlacing method has also been used to obtain bounds of the chromatic number of a graph power \cite{acfnz21}. We apply the method to derive results (Theorems~\ref{theorem:lowerboundi(G^t)} and \ref{theorem:upperboundi(G^t)}) that generalize the first and third bounds of Lemma~\ref{lemma:power} to all power graphs. Additionally, Theorems~\ref{theorem:lowerboundi(G^t)} and \ref{theorem:upperboundi(G^t)} can be seen as an extension of the lower bound of \eqref{eq:Mohar} and as an extension of \eqref{eq:QKMmun} respectively. Since the upper bound of \eqref{eq:Mohar} (corresponding to the second bound of Lemma~\ref{lemma:power}) requires the spectral partitioning method, it is not directly suitable for generalization to general graph powers through the polynomial interlacing method.

To derive our main results in this section we need the following preliminary result.

\begin{lemma}\label{lemma:interlacing}
    Let $M$ be a symmetric $n\times n$-matrix with eigenvalues $\nu_1\ge \dots \ge \nu_n$, such that $M$ has constant row sum $\nu_1$. Let $S\subseteq \{1,\dots,n\}$ be a non-empty set of indices with $1\le|S|\le \frac n2$, and write $S^c =\{1,\dots,n\} \setminus S$. Write
    $$i_M(S)=\frac 1{|S|}\sum_{u\in S} \sum_{v\in S^c} M_{uv}.$$
    The following inequality holds:
    \[
    \left ( 1-\frac {|S|}{n}\right )(\nu_1-\nu_2) \le i_M(S) \le \left (1 -\frac{|S|}{n}\right ) \left (\nu_1-\nu_n\right ).
    \]
\end{lemma}
Lemma~\ref{lemma:interlacing} is a generalization of Theorem~\ref{thm:interlacing}, after applying it to the adjacency matrix of a regular graph, in which case $i_A(S)=i_G(S)$.
\begin{proof}
    Let $B$ be the quotient matrix of $M$ with respect to the partition $\{S,S^c\}$, then the eigenvalues of $M$ interlace those of $B$. Since $M$ has constant row sum $\nu_1$, also $B$ has constant row sum $\nu_1$, and so $\nu_1$ is an eigenvalue of $B$. Let $\beta$ be the other eigenvalue of $B$.
    
    Since the sum of eigenvalues is equal to the trace, we get
    \begin{align*}
        \beta&=\tr(B)-\nu_1=B_{1,1}+B_{2,2}-\nu_1=(\nu_1-B_{1,2}) + (\nu_1- B_{2,1}) -\nu_1\\
        &=\nu_1-(B_{2,1}+B_{1,2})=\nu_1-\frac n{n-|S|} B_{1,2},
    \end{align*}
    where in the last step we use $B_{2,1}=\frac{|S|}{n-|S|} B_{1,2}$ (which follows from the definition of the quotient matrix).
    
    By interlacing, $\beta\le \nu_1$, so that $\beta$ is the second largest eigenvalue of $B$. Hence, again by interlacing, we have
    $$\nu_2\ge \beta=\nu_1-\frac n{n-|S|} B_{1,2},$$
    from which the first inequality follows, and $\nu_n\le \beta$, from which the second inequality follows.
\end{proof}

Now we are ready to show the lower and upper bounds of the isoperimetric number of a graph power in terms of the spectrum of the base graph. Throughout the proofs, we use the following key observation.
\begin{observation}\label{obs:degree-distance}
    If $p$ is any polynomial, and $u$ and $v$ are vertices of $G$ at distance more than the degree of $p$, then $p(A)_{uv}=0$. This is because $A^i_{uv}$ counts the number of walks of length $i$ between $u$ and $v$; it is zero if $u$ and $v$ have distance more than $i$.
\end{observation}

In the following, we write $\R_t[x]$ for the set of polynomials $p\in \mathbb R[x]$ of degree at most $t$.

\begin{theorem}\label{theorem:lowerboundi(G^t)}
    Let $G$ be a regular graph on $n\ge 2$ vertices with adjacency eigenvalues $\lambda_1 \ge \dots \ge \lambda_n$. Let $p\in \R_t[x]$, and define $W(p) = \max_{1\le i < j \le n} (p(A))_{ij}$ the maximum off-diagonal entry of $p(A)$. If $W(p)> 0$ and $p(\lambda_1)\ge \Lambda(p)$, then
    $$i(G^t) \ge \frac{p(\lambda_1) - \Lambda(p)}{2W(p)}.$$
\end{theorem}

\begin{proof}
    Since the all-ones vector is an eigenvector of $p(A)$ with eigenvalue $p(\lambda_1)$, we have that $p(A)$ has constant row sum $p(\lambda_1)$. By $p(\lambda_1) \ge \Lambda(p)$, the eigenvalue $p(\lambda_1)$ is the largest eigenvalue of $p(A)$, so that we can apply Lemma~\ref{lemma:interlacing} to $p(A)$. Note that the second largest eigenvalue of $p(A)$ is $\Lambda(p)$.
    
    Let $S$ be a subset of vertices of $G$ such that $|S|\le \frac n2$ and
    $$i(G^t)=i_{G^t}(S)=\frac{|\partial_{G^t}(S)|}{|S|}.$$
    The first inequality of Lemma~\ref{lemma:interlacing} applied to $M=p(A)$ and $S$, gives
    $$i_{p(A)}(S) \ge \frac{p(\lambda_1)-\Lambda(p)}2.$$
    To complete the proof, note that $W(p)>0$ and that
    $$i_{p(A)}(S) = \frac 1{|S|} \sum_{u\in S}\sum_{v\in S^c} (p(A))_{uv} \le W(p)\cdot \frac{|\partial_{G^t}(S)|}{|S|}= W(p) i_{G^t}(S)=W(p)i(G^t)$$
    since $(p(A))_{uv}=0$ if $d(u,v)> t$ by Observation~\ref{obs:degree-distance}, and $(p(A))_{uv}\leq W(p)$ otherwise, which is precisely when $(u,v)$ is part of the edge cut $(S,S^c)$ in $G^t$.
\end{proof}

Note that the first inequality of Lemma \ref{lemma:power}, when $\deg(p)\le t$, also follows from Theorem \ref{theorem:lowerboundi(G^t)}, for any polynomial $p$ with $p(A)=A(G^t)$ we have $W(p)=1$.

\begin{theorem}\label{theorem:upperboundi(G^t)}
    Let $G$ be a non-empty regular graph on $n\ge 2$ vertices with adjacency eigenvalues $\lambda_1\ge \dots \ge \lambda_n$. Let $p\in \R_t[x]$, and define
    $$w(p)=\min\left\{(p(A))_{uv}\colon\,1\le d(u,v)\le t\right\}.$$
    If $w(p)>0$, then
    $$i(G^t) \le \frac{\lceil \frac n2 \rceil}n \cdot \frac{p(\lambda_1)-\lambda(p)}{w(p)}.$$
\end{theorem}

\begin{proof}
    Since the all-ones vector is an eigenvector of $p(A)$ with eigenvalue $p(\lambda_1)$, we have that $p(A)$ has constant row sum $p(\lambda_1)$. In order to apply Lemma~\ref{lemma:interlacing}, we claim that this row sum $p(\lambda_1)$ is the largest eigenvalue of $p(A)$. Note that this statement does not change if we change the constant coefficient of $p$, so without loss of generality, we may assume that $p(A)$ has only non-negative values on the diagonal. Also off-diagonal, $p(A)$ is non-negative: if $u\ne v$ and $d(u,v) \le t$, then $p(A)_{uv} \ge w(p)>0$, and if $d(u,v) >t$, then $p(A)_{uv} =0$ by Observation~\ref{obs:degree-distance}. Let $\nu$ be any eigenvalue of $p(A)$ with eigenvector $x$, and let $j$ be such that $|x_j|$ is maximal. We get
    $$| \nu x_j|=\left | \sum_{i} p(A)_{ij} x_i \right |\le \sum_{i} p(A)_{ij} |x_i|\le \sum_{i} p(A)_{ij}|x_j|=p(\lambda_1) |x_j|,$$
    so $\nu \le |\nu|\le p(\lambda_1)$. We can now apply Lemma~\ref{lemma:interlacing} to $p(A)$.

    Let $S$ be any subset of vertices of $G$ with $|S|=\lfloor \frac n2 \rfloor$. The second inequality of Lemma~\ref{lemma:interlacing} applied to $M=p(A)$ and $S$, says
    $$i_{p(A)}(S) \le \frac {\lceil \frac n2 \rceil}n (p(\lambda_1)-\lambda(p)).$$
    To complete the proof, note that
    $$i_{p(A)}(S)=\frac 1 {|S|}\sum_{u\in S} \sum_{v\in S^c} (p(A))_{uv} \ge w(p)\cdot \frac{|\partial_{G^t}(S)|}{|S|}= w(p)i_{G^t}(S)\ge w(p) i(G^t),$$
    since all entries of $p(A)$ that do not correspond to an edge in $G^t$ are zero (because the degree of $p$ is at most $t$), and all entries of $p(A)$ corresponding to an edge $(u,v)$ in the edge cut $(S,S^c)$ of $G^t$ are at least $w(p)$.
\end{proof}

The third inequality of Lemma~\ref{lemma:power} also follows from Theorem~\ref{theorem:upperboundi(G^t)}, since for such a polynomial we have $w(p)=1$.

\begin{remark}
A generalization of both Theorem~\ref{theorem:lowerboundi(G^t)} and Theorem~\ref{theorem:upperboundi(G^t)} is possible if we include polynomials of degree higher than $t$; however, we then need the additional requirement on the polynomial $p$ that for vertices $u,v$ at distance more than $t$ we have that $p(A)_{uv}$ is non-positive for Theorem~\ref{theorem:lowerboundi(G^t)}, and non-negative for Theorem~\ref{theorem:upperboundi(G^t)}. 
\end{remark}

%%%%%%%%%%%%%%%%%%%%%%%%%%%%%%%%%%%%%%%%%%%%%%%%%%%%%%%%%%%%%%%%%%%%%%%%%%%
\subsection{Optimal polynomials in Theorem \ref{theorem:lowerboundi(G^t)} and Theorem \ref{theorem:upperboundi(G^t)} when \texorpdfstring{$t=2$}{t=2}}
%%%%%%%%%%%%%%%%%%%%%%%%%%%%%%%%%%%%%%%%%%%%%%%%%%%%%%%%%%%%%%%%%%%%%%%%%%%

Theorems~\ref{theorem:lowerboundi(G^t)} and \ref{theorem:upperboundi(G^t)} are stated for general polynomials $p$. For the case $t=2$, we are able to optimize this $p$, and determine the best possible lower and upper bounds on the isoperimetric number of the graph square $G^2$ as taken from Theorems~\ref{theorem:lowerboundi(G^t)} and \ref{theorem:upperboundi(G^t)}.

\begin{theorem}\label{theorem:lowerboundi(G^2)closed}
    Let $G$ be a non-complete regular connected graph with adjacency eigenvalues $\lambda_1\ge \dots \ge \lambda_n$. Let $\Lambda$ be the maximum number of common neighbours of adjacent vertices of $G$, and let $M$ be the maximum number of neighbours of non-adjacent vertices of $G$. The following choices of $p$ are optimal for the bound of Theorem \ref{theorem:lowerboundi(G^t)} when $t=2$.
    \begin{enumerate}
        \item[i.] If $\Lambda-\lambda_n \le M+\lambda_2$, then $p=x^2+(M-\Lambda)x$, giving
        $$i(G^2) \ge \frac{(\lambda_1-\lambda_2)(\lambda_1+\lambda_2+M-\Lambda)}{2M}.$$
        \item[ii.] If $M+\lambda_2 \le \Lambda - \lambda_n \le \lambda_1$, then $p=x^2+(M-\Lambda)x$, giving
        $$i(G^2)\ge \frac{(\lambda_1-\lambda_n)(\lambda_1+\lambda_n+M-\Lambda)}{2M}.$$
        \item[iii.] If $\max(\lambda_1,M+\lambda_2) \le \Lambda-\lambda_n$, then $p=x^2-(\lambda_2+\lambda_n)x$, giving
        $$i(G^2)\ge \frac{(\lambda_1-\lambda_2)(\lambda_1-\lambda_n)}{2(\Lambda-\lambda_2-\lambda_n)}.$$
    \end{enumerate}
\end{theorem}

\begin{proof}
    Let $p=ax^2+bx+c \in \R_2[x]$. Let $B(p)=\frac{p(\lambda_1) - \Lambda(p)}{2W(p)}$ denote the bound from Theorem \ref{theorem:lowerboundi(G^t)}. We optimize $B(p)$ for the parameters $a$, $b$, and $c$, provided that $p(\lambda_1)\ge \Lambda(p)$ and $W(p)>0$. First of all, note that scaling by a positive number and shifting does not change the value of $B(p)$, so that we may assume without loss of generality that $c=0$ and $a\in \{-1,0,1\}$. We make a case distinction according to the value of $a$.
    \begin{itemize}
    \item If $a=0$, then by $W(p)>0$ and scaling, we may assume that $b=1$, and so $p=x$ and $B(p)=(\lambda_1-\lambda_n)/2$.
    \item If $a=-1$, then $p=-x^2+bx$. We show that this case gives bounds that are worse than the bound from the case $a=0$. To show this, we investigate the numerator and denominator of $B(p)$.
    
    For the numerator of $B(p)$ we note the following. For any $1\le i\le j\le n$, we have that $p(\lambda_i)\ge p(\lambda_j)$ if and only if $b\ge \lambda_i+\lambda_j$ or $\lambda_i=\lambda_j$. Since $G$ is connected, $\lambda_1>\lambda_2$ by the Perron-Frobenius Theorem \cite{spectra}. Thus, in order for $p(\lambda_1)\ge \Lambda(p)$, by the above we need $b\ge \lambda_1+\lambda_2$, and consequently $\Lambda(p)=p(\lambda_2)$, again by the above. Now $p(\lambda_1)-\Lambda(p)=p(\lambda_1)-p(\lambda_2)=(\lambda_1-\lambda_2)(b-\lambda_1-\lambda_2)$.

    For the denominator of $B(p)$, we have the following. Write $\eta$ for the minimum number of common neighbours of two adjacent vertices of $G$. In order to have $W(p)>0$ we need $b>\eta$ because $A^2$ only has non-negative entries. In this case $W(p)=b-\eta$.
    So, for $p(\lambda_1)\ge \Lambda(p)$ and $W(p)>0$, we need that $b\ge\lambda_1+\lambda_2$ and $b>\eta$, in which case
    $$B(p)=\frac{(\lambda_1-\lambda_2)(b-\lambda_1-\lambda_2)}{2(b-\eta)}=\frac{\lambda_1-\lambda_2}2 \left (1+\frac{\eta-\lambda_1-\lambda_2}{b-\eta} \right).$$
    Note that $\eta < \lambda_1-1 < \lambda_1+\lambda_2$ since $G$ is not complete, so $B(p)< (\lambda_1-\lambda_2)/2$, implying that the bounds coming from polynomials in this case are worse than the bound from the polynomial $p=x$ as in the above case $a=0$.
    
    \item If $a=1$, then $p=x^2+bx$. We again investigate the numerator and denominator of $B(p)$.

    For the numerator of $B(p)$, again if $1\le i \le j \le n$, then $p(\lambda_i)\ge p(\lambda_j)$ if and only if $b\ge -\lambda_i-\lambda_j$ or $\lambda_i=\lambda_j$. In order to have $p(\lambda_1)\ge \Lambda(p)$, we therefore need that $b\ge -\lambda_1-\lambda_n$. Furthermore, by convexity of the parabola $p=x^2+bx$, we have $\Lambda(p)\in \left\{p(\lambda_2),p(\lambda_n)\right\}$, and if $b\ge -\lambda_2-\lambda_n$, then $\Lambda(p)=p(\lambda_2)$. If instead $-\lambda_1-\lambda_n \le b \le -\lambda_2-\lambda_n$, then $\Lambda(p)=p(\lambda_n)$. Summarizing, we have 
    \[p(\lambda_1)-\Lambda(p)=\begin{cases}
        $negative$ &$if $b<-\lambda_1-\lambda_n,\\
        (\lambda_1-\lambda_n)(b+\lambda_1+\lambda_n) &$if $-\lambda_1-\lambda_n \le b \le -\lambda_2 - \lambda_n,\\
        (\lambda_1-\lambda_2)(b+\lambda_1+\lambda_2) &$if $-\lambda_2-\lambda_n \le b.
    \end{cases}\]

    For the denominator of $B(p)$, note that $W(p)=\max(\Lambda+b,M)$. Since $G$ is non-complete and connected, $M>0$, so $W(p)>0$ is automatically satisfied. We now have
    \[W(p)=\begin{cases}
        M &$if $b\le M-\Lambda,\\
        \Lambda+b &$if $M-\Lambda\le b.
    \end{cases}\]

    To optimize the value of $B(p)$, we look at the behaviour of $B(p)$ on three (possibly empty) key intervals. Recall that $b \ge -\lambda_1-\lambda_n$.
    \begin{align*}
        I_1 &= \Big[-\lambda_1-\lambda_n,\ M-\Lambda\Big ],\\
        I_2 &= \Big [\max\big(-\lambda_1-\lambda_n,\ M-\Lambda\big),\ -\lambda_2-\lambda_n\Big ],\\
        I_3 &= \Big [ \max\big(-\lambda_2-
        \lambda_n,\ M-\Lambda\big),\ \infty \Big ).
    \end{align*}
    \begin{enumerate}
    \item For $b\in I_1$, the denominator of $B(p)$ is constant and equal to $2M$. The numerator is increasing in $b$, so $B(p)$ is increasing on $I_1$.
    \item If $b\in I_2$, then
    $$B(p)=\frac{(\lambda_1-\lambda_n)(b+\lambda_1+\lambda_n)}{2(b+\Lambda)}=\frac{\lambda_1-\lambda_n}{2}\left (1+\frac{\lambda_1+\lambda_n-\Lambda}{b+\Lambda}\right ).$$
    So $B(p)$ is strictly increasing on $I_2$ if $\Lambda>\lambda_1+\lambda_n$, constant if $\Lambda=\lambda_1+\lambda_n$, and strictly decreasing if $\Lambda< \lambda_1+\lambda_n$.
    \item If $b\in I_3$, then
    $$B(p)=\frac{(\lambda_1-\lambda_2)(b+\lambda_1+\lambda_2)}{2(b+\Lambda)}=\frac{\lambda_1-\lambda_2}2 \left (1+\frac{\lambda_1+\lambda_2-\Lambda}{b+\Lambda} \right).$$
    By regularity and non-completeness $\Lambda\le \lambda_1-1<\lambda_1+\lambda_2$, so that $B(p)$ is decreasing on $I_3$.
    \end{enumerate}
    To complete the proof, we now make the following case distinction.
    \begin{itemize}
        \item If $-\lambda_2-\lambda_n \le M-\Lambda$, then $I_2$ is empty. Since $B(p)$ is increasing on $I_1=[-\lambda_1-\lambda_n,\ M-\Lambda]$ and decreasing on $I_3=[M-\Lambda,\ \infty)$, the optimal value for $b$ is $M-\Lambda$, giving the bound
        $$i(G^2)\ge B(p) = \frac{(\lambda_1-\lambda_2)(\lambda_1+\lambda_2+M-\Lambda)}{2M},$$
        which is greater than $B(x)$ because $\Lambda <\lambda_1+\lambda_2$.
            \item If $M-\Lambda\le -\lambda_2-\lambda_n$ and $\lambda_1+\lambda_n \le \Lambda$, then $B(p)$ is increasing on $I_1\cup I_2=[-\lambda_1-\lambda_n,\ -\lambda_2-\lambda_n]$, and decreasing on $I_3=[-\lambda_2-\lambda_n,\ \infty)$. It follows that the optimal value of $b$ is $-\lambda_2-\lambda_n$, giving the bound
            $$i(G^2)\ge B(p) = \frac{(\lambda_1-\lambda_2)(\lambda_1-\lambda_n)}{2(\Lambda-\lambda_2-\lambda_n)},$$
            which is again greater than $B(x)$ because $\Lambda<\lambda_1+\lambda_2$.
            \item Lastly, if $M-\Lambda \le -\lambda_2-\lambda_n$ and $\Lambda \le \lambda_1+\lambda_n$, then we must have $-\lambda_1-\lambda_n \le M-\Lambda$, and so $B(p)$ is increasing on $I_1=[-\lambda_1-\lambda_n,\ M-\Lambda]$, and decreasing on $I_2\cup I_3=[M-\Lambda,\ \infty)$. Now the optimal value of $b$ is $M-\Lambda$, giving the bound
            $$i(G^2)\ge B(p)=\frac{(\lambda_1-\lambda_n)(\lambda_1+\lambda_n+M-\Lambda)}{2M}.$$
            Again, this is greater than $B(x)$, because $\Lambda\le \lambda_1+\lambda_n$ and $\lambda_1-\lambda_2 \le \lambda_1-\lambda_n$.
    \end{itemize}
    \end{itemize}
    This concludes the proof.
\end{proof}

For the best possible upper bound, we show the following result, which may be of independent interest.
 
\begin{proposition}\label{proposition:nu<=}
    Let $G$ be a non-empty graph with adjacency eigenvalues $\lambda_1\ge \dots \ge \lambda_n$. If $\eta$ is the minimum number of common neighbours of two adjacent vertices, then $\eta \le \lambda_1+\lambda_n$.
\end{proposition}

\begin{proof}
    Suppose, by contradiction, that $\eta > \lambda_1+\lambda_n$.

    First, assume that $G$ is connected. Consider the polynomial $p(x)=x^2-(\lambda_1+\lambda_n)x$, and note that $p(\lambda_1)=p(\lambda_n)$. Note that $p(A)$ is a non-negative matrix: if $u=v$, then $p(A)_{uv}=\deg(u)>0$, if $u\ne v$ and $u\not \sim v$, then $p(A)_{uv}=(A^2)_{uv}\ge 0$, and if $u\sim v$, then $p(A)_{uv}=(A^2)_{uv}-(\lambda_1+\lambda_n)\ge \eta -\lambda_1-\lambda_n >0$. Furthermore, note that $p(A)^m$ is a positive matrix if $m$ is at least half the diameter of $G$, so that $p(A)$ is primitive. By the Perron-Frobenius Theorem \cite{spectra}, the eigenvalue $p(\lambda_1)$ for the all-ones eigenvector of $p(A)$ must be strictly larger in absolute value than any other eigenvalue of $p(A)$. In particular $p(\lambda_1)>p(\lambda_n)$, a contradiction.

    Next, if $G$ is disconnected, then let $C$ be a connected component of $G$ that has smallest eigenvalue $\lambda_n$. Such a connected component exists because the spectrum of $G$ is equal to the disjoint union of the spectra of its connected components. Note that the largest eigenvalue $\lambda_{\max}(C)$ of $C$ is bounded above by $\lambda_1$. Applying the above argument to $C$, we get that there are two adjacent vertices in $C$ that have at most $\lambda_{\max}(C)+\lambda_n\le \lambda_1+\lambda_n$ common neighbours. The result now follows for $G$.
\end{proof}

\begin{theorem}\label{theorem:upperboundi(G^2)closed}
    Let $G$ be a regular graph with adjacency eigenvalues $\lambda_1\ge \dots \ge \lambda_n$ that is not a disjoint union of complete graphs. Let $\eta$ be the minimum number of common neighbours of adjacent vertices, and let $\xi$ be the minimum number of common neighbours of two vertices at distance 2 in $G$. Let $\lambda_i$ be the eigenvalue of $G$ that is closest to $(\eta-\xi)/2$, then the optimal polynomial $p$ from Theorem~\ref{theorem:upperboundi(G^t)} is $p(x)=x^2+(\xi-\eta)x$, giving
    $$i(G^2) \le \frac{\lceil \frac n 2 \rceil}{n} \cdot \frac{(\lambda_1-\lambda_i)(\lambda_1+\lambda_i+\xi-\eta)}{\xi}.$$
\end{theorem}

\begin{proof}
In order to prove this theorem, we use Proposition~\ref{proposition:nu<=}. Let $p=ax^2+bx+c \in \R_2[x]$. Write $B(p)=\frac{p(\lambda_1)-\lambda(p)}{w(p)}$ so that the bound from Theorem~\ref{theorem:upperboundi(G^t)} is $\lceil n/2 \rceil B(p)/n$. We minimize $B(p)$ for the parameters $a$, $b$, and $c$, provided that $w(p)>0$. First of all, note that scaling by a positive number and shifting does not change the value of $B(p)$, so that we may assume without loss of generality that $c=0$ and $a\in \{-1,0,1\}$. However, if $a\le 0$, then let $u$ and $v$ be two vertices at distance 2 (which exist because $G$ is not a disjoint union of complete graphs), then $w(p)\le p(A)_{uv}=a\cdot |N(u) \cap N(v)| \le 0$, which is not allowed. Therefore, we only have to consider polynomials of the form $p=x^2+bx$ and minimize $B(p)$ for the parameter $b$.

    For the denominator $w(p)$ of $B(p)$, note that
    \[
    w(p)=\begin{cases}
        \eta+b &$if $b\le\xi-\eta,\\
        \xi &$if $\xi -\eta \le b.
    \end{cases}
    \]
    In order for $w(p)>0$, we therefore need that $b> -\eta$.

    For the numerator $p(\lambda_1)-\lambda(p)$ of $B(p)$, note that since $p(x)=x^2+bx$ defines a convex parabola with minimum attained at $-b/2$, we know that $\lambda(p)$ is equal to $p(\lambda_j)$ such that $\lambda_j$ is the eigenvalue of $G$ that is closest to $-b/2$. By Proposition~\ref{proposition:nu<=} and $b> -\eta$, we know that $-b/2< (\lambda_1+\lambda_n)/2$, so that $\lambda_n$ is closer to $-b/2$ than $\lambda_1$, and so $j\ne 1$. Write $\theta_0>\theta_1>\dots>\theta_d$ for the $d+1$ distinct eigenvalues of $G$, and consider the intervals $I_j$ (for $j=1,\dots,d$) defined by
    \[
    I_j =\begin{cases}
        [-(\theta_{j-1}+\theta_j),-(\theta_j+\theta_{j+1})] &$if $1\le j \le d-1,\\
        [-(\theta_{d-1}+\theta_d),\infty) &$if $j=d.
    \end{cases}
    \]
    Then $b\in I_j$ if and only if $\lambda(p)=p(\theta_j)$. On the interval $I_j$, the numerator of $B(p)$ is equal to
    $$p(\theta_0)-\lambda(p)=p(\theta_0)-p(\theta_j)=(\theta_0-\theta_j)(\theta_0+\theta_j+b).$$
    It follows that $p(\theta_0)-\lambda(p)$ is increasing in $b$.

    We now show that $B(p)$ is decreasing for $b\in (-\eta , \xi-\eta]$ and increasing for $b\in [\xi-\eta,\infty)$, so that the optimal value for $b$ is $\xi-\eta$. If $b\ge \xi-\eta$, then the denominator of $B(p)$ is constant, while the numerator of $B(p)$ is increasing in $b$, so that $B(p)$ is increasing in $b$.

    Otherwise, for $-\eta < b \le \xi -\eta$ we have $w(p)=\eta+b$. Let $j\in\{1,\dots,d\}$. If $I_j$ intersects $(-\eta,\xi-\eta]$, then on this intersection we have
    $$B(p)=\frac{(\theta_0-\theta_j)(b+\theta_0+\theta_j)}{b+\eta}=(\theta_0-\theta_j) \left ( 1+\frac{\theta_0+\theta_j-\eta}{b+\eta} \right ).$$
    By Proposition~\ref{proposition:nu<=} this is a function decreasing in $b$. The result follows.
\end{proof}

%%%%%%%%%%%%%%%%%%%%%%%%%%%%%%%%%%%%%%%%%%%%%%%%%%%%%%%%%%%%%%%%%%%%%%%%%%%
\subsection{Optimization of the bounds in Theorem \ref{theorem:lowerboundi(G^t)} and Theorem \ref{theorem:upperboundi(G^t)}}
%%%%%%%%%%%%%%%%%%%%%%%%%%%%%%%%%%%%%%%%%%%%%%%%%%%%%%%%%%%%%%%%%%%%%%%%%%%
Next, we formulate Theorem~\ref{theorem:lowerboundi(G^t)} as a linear program which finds an optimal polynomial for any value of~$t$. Write again $\theta_0>\theta_1>\dots>\theta_d$ for the $d+1$ distinct eigenvalues of $G$, note that $\theta_0=\lambda_1$. We may scale and translate~$p$ without changing the value of the bound. Assume therefore without loss of generality that~$p(\theta_0)=1$ and~$\Lambda(p)=0$. Note that~$W(p)>0$ by assumption, so the scaling does not flip the sign of the bound. We then loosen the restriction to $\Lambda(p) \leq 0$, which gives us LP \eqref{eq:iso_LP_lower}. Before we explain why we do not lose anything by changing this constraint to an inequality, we will first explain how to interpret LP \eqref{eq:iso_LP_lower}. We have variables $a_j$, which are the coefficients of a polynomial $p(x) = \sum_{j=0}^t a_jx^j$, and a variable $M$. By fixing $p(\theta_0)=1,\Lambda(p)=0$, the problem of maximizing $\frac{p(\theta_0)-\Lambda(p)}{2W(p)}$ is equivalent to minimizing $W(p)$. We note that $W(p)$ is a maximum of $p(A)_{uv}$ over all pairs of distinct vertices $u,v \in V$, hence we introduced the variable $M$ which represents $W(p)$. The first set of constraints then force $M \geq W(p)$, the second constraint forces $p(\theta_0) =1$, and the last set of constraints force $\Lambda(p) \leq 0$. 

Let us now explain why loosening the constraint $\Lambda(p)=0$ to $\Lambda(p) \leq 0$ does not change the optimum of the optimization problem. LP \eqref{eq:iso_LP_lower} has $t+2$ variables, meaning there exists an optimal solution for which $t+2$ independent constraints are tight. Since $A^0 = I$, the set of constraints of the form $\sum_{j = 0}^t a_j \cdot (A^i)_{uv}$ contains at most $t$ independent constraints. In addition to the tight constraint $\sum_{j = 0}^t a_j \cdot \theta_0^j = 1$, we are left with one more constraint which can be assumed to be tight. This will be of the form $ \sum_{j = 0}^t a_j \cdot \theta_i^j \leq 0$, and hence we find that LP \eqref{eq:iso_LP_lower} has an optimal solution with $\Lambda(p)= 0$. In conclusion, loosening the constraint $\Lambda(p)=0$ to $\Lambda(p) \leq 0$ does not affect the optimal value, so LP \eqref{eq:iso_LP_lower} indeed gives us the best possible bound that can be obtained by Theorem \ref{theorem:lowerboundi(G^t)}.

\begin{equation}
\label{eq:iso_LP_lower}
\boxed{
\begin{array}{rll}
{\tt minimize} & M \\
{\tt subject\ to} & \sum_{j = 0}^t a_j \cdot (A^j)_{uv} \leq M,&  u,v \in V, u \neq v\\
& \sum_{j = 0}^t a_j \cdot \theta_0^j = 1 \\
 & \sum_{j = 0}^t a_j \cdot \theta_i^j \leq 0, &  i = 1,\dots,d\\
 & M \geq 0
\end{array}}
\end{equation}

For Theorem \ref{theorem:upperboundi(G^t)} we can apply the same tricks. We assume $p(\theta_0)= 1, \lambda(p) = 0$, and then we can loosen the second constraint to $\lambda(p) \geq 0$, by the same argument as in the LP formulation for Theorem \ref{theorem:lowerboundi(G^t)}. The structure of the LP is very similar as well. The variables $a_j$ represent the coefficients of a polynomial, and the variable $M$ represents $w(p)$. The first set of constraints then force $M \leq w(p)$, the second constraint forces $p(\theta_0) =1$, and the last set of constraints force $\lambda(p) \geq 0$.

\begin{equation}
\label{eq:iso_LP_upper}
\boxed{
\begin{array}{rll}
{\tt maximize} & M \\
{\tt subject\ to} & \sum_{j = 0}^t a_j\cdot (A^j)_{uv} \geq M,&  u,v \in V, 1\leq d(u,v)\leq t\\
& \sum_{j = 0}^t a_j\cdot \theta_0^j = 1\\
& \sum_{j = 0}^t a_j\cdot \theta_i^j \geq 0, &  i = 1,\dots,d\\
& M \geq 0
\end{array}}
\end{equation}
    In Appendix \ref{sec:app_sims} we test the performance of Theorems \ref{theorem:lowerboundi(G^t)} and \ref{theorem:upperboundi(G^t)} for several graphs in the database of Sagemath. 
    In Table \ref{tab:t2_sim} we consider the case $t=2$, using the closed expressions from Theorem \ref{theorem:lowerboundi(G^2)closed} and Theorem \ref{theorem:upperboundi(G^2)closed}. In Tables \ref{tab:t3_sim} and \ref{tab:t4_sim}, we consider the cases $t=3$ and $t=4$ respectively; for this, we use LP \eqref{eq:iso_LP_lower} and LP \eqref{eq:iso_LP_upper}.
    Distance-regularity of the considered graphs  is indicated in the tables (DRG). These computations were done on a laptop equipped with a 12th Gen Intel Core i7-12700H processor and 16 GB RAM, running Ubuntu 22.04, SageMath 10.7. If a computation takes more than 5 minutes to complete, the table entry is instead filled with ``time".
%%%%%%%%%%%%%%%%%%%%%%%%%%%%%%%%%%%%%%%%%%%%%%%%%%%%%%%%%%%%%%%%%%%%%%%%%%%
\section{Approximating the isoperimetric number of distance-regular graphs}\label{sec:DRG}
%%%%%%%%%%%%%%%%%%%%%%%%%%%%%%%%%%%%%%%%%%%%%%%%%%%%%%%%%%%%%%%%%%%%%%%%%%%
In this section we focus on distance-regular graphs and specifically on lower bounds for the sparsity of distance-regular graphs. The sparsity of a graph can be used as a $2$-approximation of the (normalized) isoperimetric number. In particular, one sees that 
\begin{equation}\label{eq:sparsityisoperimetric}
    \frac{1}{2} \sigma(G) \leq \frac{i(G)}{n} \leq \sigma(G). 
\end{equation}

Upper bounds on the isoperimetric number of distance-regular graphs were studied in \cite{KL2013, QKM2020}. In \cite{QKM2020} the authors conjecture that $i(G) \le \mu_2$ for all distance-regular graphs. This, combined with the lower bound~\eqref{eq:Mohar}, would imply that $\mu_2$ is a $2$-approximation for the isoperimetric number. They prove the bound for strongly regular graphs, that is, distance-regular graphs with diameter $D = 2$. A more involved upper bound for distance-regular graphs, using limits of Green functions, is given in \cite{KL2013}. 

An alternative lower bound for the isoperimetric number comes from a linear program relaxation for the sparsity~\cite{LLR1995}, see LP~\eqref{eq:LP_Linial}. For strongly regular graphs, a closed formula for the optimal value of LP~\eqref{eq:LP_Linial} was found in~\cite{S2005} (note that this paper uses ``isoperimetric number'' to denote what we call the sparsity of $G$). The main contribution of this section is the extension of this approach to all distance-regular graphs, implying that the optimal value of the linear program can be computed using only the intersection array, without having to construct the distance-regular graph. Let $G$ be a distance-regular graph of diameter $D$ and let $k_j$ denote the number of vertices at distance $j$ from any given vertex. The lower bound for the sparsity that we obtain is

\begin{equation}\label{eq:sparsity drg}
\sigma(G) \geq \frac{k_1}{\sum_{j=1}^D jk_j},
\end{equation}
see Corollary~\ref{cor:sparsity_drg}.

As a byproduct (see Corollary~\ref{cor:drg_approx}), we show that taking as $S$ a single vertex gives a $2D$-approximation of the isoperimetric  number of distance-regular graphs.

In \cite{DM2019} the authors use another LP relaxation to study the minimum cut size between any $S$ and $S^c$ with $|S|$ fixed. This is the so-called \emph{CR$^\star$ bound}. From this relaxation, one obtains a lower bound for the isoperimetric number that improves on the lower bound in \ref{eq:Mohar}. We show that our bound is incomparable with these lower bounds.\break

We assume that $G$ has vertex set $\{1,2,\ldots, n\}$. Denote by $\text{LP}(G)$ the objective value of the following linear program~\cite{LLR1995}:

\begin{equation}\label{eq:LP_Linial}
\boxed{
\begin{array}{rll}
{\tt minimize} & \sum_{u<v, u \sim v} \, x_{uv} \\
{\tt subject\ to} & \sum_{u < v} x_{uv} = 1 \\
& x_{uw} + x_{vw} - x_{uv} \geq 0\\
& x_{uv} + x_{vw} - x_{uw} \geq 0, & \forall u,v,w : u<v<w\\
& x_{uv} + x_{uw} - x_{vw} \geq 0\\
& x_{uv} \geq 0, & \forall u,v: u<v\\
\end{array}}
\end{equation}

Let $G$ be a graph and let $S \subseteq V$ such that $\sigma(S) = \sigma(G)$. Consider the assignment $x_{uv} = (|S||V\setminus S|)^{-1}$ if $\{u, v\} \in \partial_G(S)$ and $0$ otherwise. Note that such an assignment is feasible for LP~\eqref{eq:LP_Linial} and has objective value $\sigma(S)$. Hence, $\text{LP}(G) \leq \sigma(G)$.

Let $W_G$ denote $\sum_{u < v} d(u,v)$, the \emph{Wiener index} of the graph.
If $G$ is a distance-regular graph with diameter $D=2$, that is, a strongly regular graph, then the optimal solution for LP~\eqref{eq:LP_Linial} is achieved by taking $x_{uv}=d(u, v)/W_G$~\cite{S2005}.

We show that $x_{uv} = d(u,v)/W_G$ achieves the optimal value for all distance-regular graphs. Note that this assignment is feasible and 
gives \eqref{eq:sparsity drg}. Following the strategy from \cite{S2005}, we show that the dual of LP~\eqref{eq:LP_Linial} admits a feasible solution with the same objective value. 

Let $A$ be the adjacency matrix of $G$. As stated in \cite{S2005} the dual of LP~\eqref{eq:LP_Linial} is:

\begin{equation}\label{eq:LP_Linial_Dual}
\boxed{
\begin{array}{rll}
{\tt maximize} & \psi \\
{\tt subject\ to} &\sum_{w \in V \setminus \{u, v\}} \,y_{uw}^v + y_{vw}^u - y_{uv}^w + \psi \leq A_{uv}, & \forall u,v : u < v\\
    &y_{uv}^w \geq 0, & \forall u,v,w : u\neq v \neq w\\
\end{array}}
\end{equation}

Here $\psi$ corresponds to the $= 1$ constraint and $y^u_{vw}$ corresponds to the $x_{uv} + x_{uw} - x_{vw} \ge 0$ constraint.

Consider the following restriction on the variables of LP~\eqref{eq:LP_Linial_Dual}. Choose $y_{uv}^w$ to be zero unless $w$ is on a shortest path from $u$ to $v$, in which case its value only depends on $d(u,v)$. Distance-regularity of $G$ implies that the constraint for each pair $u,v$ depends only on $d(u,v)$. Hence this LP can be written using a condition and variable, $y_h$, for each $h = d(u,v)$. Note that $y_1$ does not appear in any of the constraints, because there can be no vertex on the path $uv$ if $d(u,v) = 1$. Additionally, we look for a solution that satisfies the inequalities involving $\psi$ with equality. Thus, our restricted LP is

\begin{equation}\label{eq:LP_Linial_Dual_Restricted}
\boxed{
\begin{array}{rll}
{\tt maximize} & \psi \\
{\tt subject\ to} &\psi =  1 -2 \sum_{i=2}^{D} \, p_{i,i-1}^1y_i \\
&\psi =  \sum_{i=1}^{h-1} \, p_{i,h-i}^h y_h -2 \sum_{i=h+1}^{D}\;  p_{i,i-h}^h y_i, & \forall h \in \{2, \dots, D\}\\
    &y_h \geq 0, & \forall h \in \{2, \dots, D\}\\
\end{array}}
\end{equation}

We have separated the first constraint, which corresponds to $u,v$ being adjacent. Since we are adding restrictions to a maximization problem, LP~\eqref{eq:LP_Linial_Dual_Restricted} gives a lower bound on LP~\eqref{eq:LP_Linial_Dual}, which in turn gives a lower bound on the sparsity.

Before finding the solution to LP~\eqref{eq:LP_Linial_Dual_Restricted}, we make one further observation about the intersection numbers of distance-regular graphs.

\begin{lemma}\label{lem:drgfactori}
    For any distance-regular graph with diameter $D$ and for all $h\in\{2,\dots,D\}$, it holds that $$\sum_{i = 1}^{h-1} 2ik_ip_{h-i,h}^i = hk_h\sum_{i = 1}^{h-1} p_{i,h-i}^h.$$
 
\end{lemma}

\begin{proof}
    We first prove that $k_ip_{h-i,h}^i = k_{h-i}p_{i,h}^{h-i}$ for all $1 \leq i \leq h-1 \leq D-1$ by a double counting of the set $$\{(u,v,w) \in V^3\colon\, d(u,v) = i,\, d(v,w) = h-i \text{ and } d(u,w) = h\}.$$ There are $n$ possible values for $u$. By definition of $k_i$ and $p_{h-i,h}^i$, there are $k_i$ options for $v$ given $u$ and there are $p_{h-i,h}^i$ options for $w$ given $u$ and $v$. For the other counting, we first fix $w$. There are $n$ options for $w$, $k_{h-i}$ options for $v$ given $w$ and $p_{i,h}^{h-i}$ options for $u$ given $v$ and $w$. Hence $nk_ip_{h-i,h}^i = nk_{h-i}p_{i,h}^{h-i}$. From this it follows that
    \begin{equation}\label{eq:sumintersectionnumbers}
    \begin{aligned}
        \sum_{i = 1}^{h-1} 2ik_ip_{h-i,h}^i &= \sum_{i = 1}^{h-1} ik_ip_{h-i,h}^i + \sum_{i = 1}^{h-1} ik_{h-i}p_{i,h}^{h-i}\\ 
        &= \sum_{i = 1}^{h-1} ik_ip_{h-i,h}^i + \sum_{i = 1}^{h-1} (h-i)k_ip_{h-i,h}^{i}\\
        &= h\sum_{i = 1}^{h-1} k_ip_{h-i,h}^i. 
    \end{aligned}            
    \end{equation}

    By double counting the set 
    $$\{(u,v,w) \in V^3\colon\, d(u,v) = h\text{ and } \exists\, i \colon\, d(v,w) = i\text{ and } d(u,w) = h-i\}.$$ we deduce analogously that 
    \begin{equation*}
        nk_h\sum_{i = 1}^{h-1} p_{i,h-i}^h = n\sum_{i = 1}^{h-1} k_ip_{h-i,h}^i
    \end{equation*}
    It follows directly that
    \begin{align*}
        hk_h\sum_{i = 1}^{h-1} p_{i,h-i}^h &= h\sum_{i = 1}^{h-1} k_ip_{h-i,i}^i\\
        &= \sum_{i = 1}^{h-1} 2ik_ip_{h-i,h}^i
    \end{align*}
    where the second equality follows from the identity~\eqref{eq:sumintersectionnumbers}.
\end{proof}

\begin{theorem}\label{thm:opt_lp_drg}
    Let $G$ be a distance-regular graph with intersection numbers $p_{ij}^h$ and with $k_i$ the number of vertices at distance $i$ from a vertex. LP~\eqref{eq:LP_Linial_Dual_Restricted} has a solution with value $\frac{k_1}{\sum_{j=1}^D jk_j}$.
\end{theorem}

\begin{proof}
    For $j \in \{1,\ldots D\}$ let $C_j$ denote the $j$th constraint of LP~\eqref{eq:LP_Linial_Dual_Restricted}, specifically for $j  \ge 2$ $C_j$ denotes
    $$\psi = \sum_{i=1}^{j-1} p_{i,j-i}^j y_j -2 \sum_{i=j+1}^{D} p_{i,i-j}^j y_i.$$
    We show that, for each $h$, the coefficients of $y_h$ cancel in the right hand side of the sum $\sum_{j = 1}^D jk_jC_j$. For any $h \in \{2,\ldots, D\}$ the variable $y_h$ only appears in the constraints $C_1,\ldots, C_h$. In particular, the coefficient of $y_h$ in the sum $\sum_jk_j C_j$ is $$hk_h\sum_{i = 1}^{h-1} p_{i,h-i}^h - 2\sum_{j = 1}^{h-1} jk_jp_{h-j,h}^j.$$ By Lemma~\ref{lem:drgfactori} this is equal to zero.  Hence $\sum_{j = 1}^D jk_jC_j$ induces the constraint $\sum_{j = 1}^D jk_j\psi = k_1$ and $\psi$ must be equal to the desired value. 
    
    We show by induction on $j$ that for all $j\in\{0,\dots,D-2\}$ there is a solution that satisfies $C_D,\ldots, C_{D-j}$ with all $y_h > 0$. Throughout we use that $p_{i,j-i}^j > 0$ for all \(j\in\{1,\dots,D\}\) and $i\in\{0,\dots,j\}$. This holds because there is a pair that is at distance $j$ and there must be a vertex on the path between them that has distance $i$ and $j-i$ to the end points. Note that constraint $C_D$ says that
    $$\psi = \sum_{i = 1}^{D-1} p_{i,D-i}^Dy_D,$$
    so we can choose $y_D$ to be the positive number that satisfies this equation. Suppose that for some $\ell\in\{1,\dots,D-2\}$, the $y_D,\ldots, y_{D-\ell+1} > 0$ have been chosen to satisfy constraints $C_D,\ldots, C_{D-\ell+1}$. Consider $C_{D-\ell}$ and observe that
    $$\sum_{i = 1}^{D-\ell-1} p_{i,D-\ell-i}^{D-\ell} y_{D-\ell} = \psi + 2\sum_{i = D-\ell+1}^Dp_{i,i-(D-\ell)}^{D-\ell}y_i > 0.$$
    The only variable that has not been chosen yet is $y_{D-\ell}$, so this linear equation determines it and it is positive. The variable $y_{D-\ell}$ does not appear in $C_D, \ldots, C_{D-\ell +1}$, so those constraints still hold. This completes the induction.
    
    Having defined $y_h$ for all $h\in\{2,\dots,D\}$ to be positive numbers such that $C_2,\ldots, C_D$ hold, we find that $C_1$ must also hold with equality because $\psi$ was chosen such that $\sum jk_jC_j$ holds. Hence we find a feasible solution to LP~\eqref{eq:LP_Linial_Dual_Restricted} with $\psi  = \frac{k_1}{\sum_{j = 1}^D jk_j}$.
\end{proof}

\begin{corollary}\label{cor:sparsity_drg}
    The sparsity of a distance-regular graph $G$ satisfies
    $$\sigma(G) \geq  \frac{k_1}{\sum_{j=1}^D jk_j}.$$
\end{corollary}

It follows from inequality~\eqref{eq:sparsityisoperimetric} that  we can also lower bound the isoperimetric number. 

\begin{corollary}\label{cor:isop_drg}
    If $G$ is a distance-regular graph with diameter $D$, then
\begin{equation}\label{eq:closed_lb_drg}
    i(G) \ge \frac{nk_1}{\sum_{j=1}^D 2jk_j} .
\end{equation}
\end{corollary}

These bounds are tight. For example, Theorem~\ref{thm:opt_lp_drg} and inequality~\eqref{eq:closed_lb_drg} are attained with equality for even cycles: the sparsity and the isoperimetric number of $C_{2n}$ can be upper bounded by taking a set $S$ that induces a copy of $P_n$, so $\sigma(C_{2n}) \leq 2/n^2$ and $i(C_{2n}) \leq 2/n$. This construction meets the lower bound from Corollary~\ref{cor:sparsity_drg}: 
    \begin{equation*}
     \frac{k_1}{\sum_{j=1}^n jk_j} = \frac{2}{2\binom{n}{2}+n} = \frac{2}{n^2}. \qedhere   
    \end{equation*}

In Table~\ref{tab:t5_drg}, we compare the lower bound $\mu_2/2$, the value of the CR$^\star$ bound, and the bound \eqref{eq:closed_lb_drg} on  distance-regular graphs. We find examples of graphs for which $\mu_2/2 \le CR^\star < \text{LP}(G) = i(G)$, such as the Dodecahedron graph.

As a byproduct, Corollary~\ref{cor:sparsity_drg} can be used to give a certificate in approximation algorithms for $\sigma(G) $ and $i(G)$. 

\begin{corollary}\label{cor:drg_approx}
    Let $G$ be a distance-regular graph with diameter $D$ and let $v \in V$. The set $S = \{v\}$ gives a $D$-approximation for $\sigma(G)$ and a $2D$-approximation for $i(G)$.
\end{corollary}
\begin{proof}
    Let $k$ be the degree of $G$.  Corollary~\ref{cor:sparsity_drg} states that
    \begin{equation*}
        \sigma(G) \geq \frac{k}{\sum_{j = 1}^D jk_j} \geq \frac{k}{\sum_{j=1}^D Dk_j} = \frac{k}{D(n-1)}.
    \end{equation*}
    On the other hand, consider any singleton $S = \{v\}$. Then we have
    \begin{equation*}
        \sigma(G) \leq \frac{|\partial S|}{|S||V\setminus S|} = \frac{k}{n-1}.
    \end{equation*}
    The $2D$-approximation follows from Equation~\eqref{eq:sparsityisoperimetric}.
\end{proof}

Note that Corollary~\ref{cor:drg_approx} does not hold for every regular graph. For instance, the Cartesian product $K_k \cartesianproduct K_2$ is a $k$-regular graph with diameter $2$ and $i(K_k \cartesianproduct K_2) \leq 1$, so the degree of a vertex is not a $2D$-approximation of the isoperimetric number for $k > 4$. 

Corollary~\ref{cor:drg_approx} gives an approximation certificate for algorithms that consider a single vertex. More involved algorithms can give better approximations. 

% In this light we want to mention Fiedler's algorithm (see for example \cite[Section 4.2]{Trevisan2017}). start by sorting the vertices of $G$ as $v_1,\dots,v_n$ according to the entries of a $\mu_2$-eigenvector, and then compute the minimum, over all indices $k\in\{1,\dots,n\}$, of $i(\{v_1,\dots,v_k\})$. The set $\{v_1\}$ is the first to be considered, and by applying Corollary~\ref{cor:drg_approx}, if $G$ is a distance-regular graph with diameter $D$, the set $S$ that is output by Fiedler’s algorithm, satisfies $i(S)\leq 2D\cdot i(G)$.

%%%%%%%%%%%%%%%%%%%%%%%%%%%%%%%%%%%%%%%%%%%%%%%%%%%%%%%%%%%%%%%%%%%%%%%%%%%
\section{The exact isoperimetric number and a connection with finite geometry}\label{sec:finitegeometry}
%%%%%%%%%%%%%%%%%%%%%%%%%%%%%%%%%%%%%%%%%%%%%%%%%%%%%%%%%%%%%%%%%%%%%%%%%%%

The determination of the exact value of the isoperimetric number for a given graph is in general an NP-hard  problem. The aim of this section is to find the exact value in some particular cases. First, we recall some known results, and derive the isoperimetric number for generalized Hamming graphs, connecting with previous results \cite{M1989, hypercubeedgeisopowers}. Next, we study the exact value for several new families of graphs. Using a simple but powerful observation derived from Corollary \ref{cor:tight}, we show that the isoperimetric number attains the spectral lower bound $\frac12\mu_2$ from \eqref{eq:Mohar} if and only if there exists a \textit{tight set} (of Type II, see Definition~\ref{def:tight} below) containing half the number of vertices of the graph. The link between the isoperimetric number and tight sets 
opens a door to a connection between graph expansion and finite geometry, providing a way to derive the isoperimetric number of graphs coming from geometries.
As far as we know, previously only the exact vertex-isoperimetric number was studied in geometric graphs \cite{Hui18,Price18}. For the (edge-)isoperimetric number of geometric graphs, only bounds have been studied \cite{levi}.
Tight sets in collinearity graphs of polar spaces have been widely studied in the literature, and the spectra of the related graphs are known, making these graphs natural candidates for examples of graphs where we can compute the isoperimetric number directly.

%%%%%%%%%%%%%%%%%%%%%%%%%%%%%%%%%%%%%%%%%%%%%%%%%%%%%%%%%%%%%%%%%%%%%%%%%%%
\subsection{Graph classes whose isoperimetric number is known}
%%%%%%%%%%%%%%%%%%%%%%%%%%%%%%%%%%%%%%%%%%%%%%%%%%%%%%%%%%%%%%%%%%%%%%%%%%%

Before we study the isoperimetric number in geometrical graphs using tight sets, we recall some families where the isoperimetric number is known. As far as we know, these are the disconnected graphs (where the isoperimetric number is trivially zero), complete graphs, cycles, paths, complete bipartite graphs and certain Hamming graphs \cite{M1989}, see Table~\ref{tab:exactnumber}. The number is also known for some of their images under taking complements, joins and Cartesian products \cite{Aslan2013}.

\begin{table}[H]
\renewcommand{\arraystretch}{1.5}
    \centering
    \begin{tabular}{l|c|l}
        \(G\) & \(i(G)\) & Description\\
        \hline\hline
        \(K_n\) & \(\lceil \frac{n}2\rceil\) & Complete graph on \(n\) vertices\\
        \(P_n\) & \(1/\lceil\frac{n}2\rceil\) & Path on \(n\) vertices\\
        \(C_n\) & \(2/\lfloor\frac{n}2\rfloor\) & Cycle on \(n\) vertices\\
        \(K_{m,n}\) & \(\lceil\frac{mn}2\rceil/\lfloor\frac{m+n}2\rfloor\) & Complete bipartite graph\\
        \(Q_n\) & 1 & Hypercube on \(2^n\) vertices\\
        \(H(n,q)\), \(q\) even & \(\frac{q}2\) & Hamming graph on \(q^n\) vertices
    \end{tabular}
    \caption{The isoperimetric number \(i(G)\) of some families of graphs, see \cite{M1989}.}
    \label{tab:exactnumber}
\end{table}

%%%%%%%%%%%%%%%%%%%%%%%%%%%%%%%%%%%%%%%%%%%%%%%%%%%%%%%%%%%%%%%%%%%%%%%%%%%
\subsection{Hamming graphs and other Cartesian products}\label{sec:hamming}
%%%%%%%%%%%%%%%%%%%%%%%%%%%%%%%%%%%%%%%%%%%%%%%%%%%%%%%%%%%%%%%%%%%%%%%%%%%
The Hamming graph \[H(n,q)\cong \underbrace{K_q\cartesianproduct\cdots\cartesianproduct K_q}_{n}\] has \(i(H(n,q))=\frac{q}2\) if \(q\) is even, by \cite[Theorem~5.1]{M1989}. But, also if \(q\) is odd, we can determine the isoperimetric number to be $i(H(n,q)) = \lceil\frac{q}{2}\rceil$. More generally, we have the following result: 

\begin{proposition}\label{prop:Hamming}
    \(i(K_{q_1}\cartesianproduct\cdots\cartesianproduct K_{q_n})=\min_i\lceil\frac{q_i}2\rceil\)
\end{proposition}

\begin{proof}
    By induction on \(n\). For \(n=1\), the statement is true. Suppose that it holds for \(n-1\) and assume that \(q_1\leq\dots\leq q_n\). Label the vertices of the graph \(G=K_{q_1}\cartesianproduct\cdots\cartesianproduct K_{q_n}\) with \(n\)-tuples \((x_1,\dots,x_n)\) where \(x_i\in\mathbb{Z}/n\mathbb{Z}\).
    It has been shown in \cite{Lindsey} that, for a fixed value \(m\), the number of edges within a vertex set of size \(m\) is maximized by the set \(S_m\) of the first \(m\) vertices of \(G\) in lexicographic order
    \[(x_1,\dots,x_n)<(y_1,\dots,y_n)\quad\Longleftrightarrow\quad\exists\,i\colon\,x_1=y_1,\,\dots,\,x_{i-1}=y_{i-1}\text{ and }x_i<y_i.\]
    Since the graph is regular, this set minimizes \(i_m(G)=\min \left\{\frac{|\partial S|}{|S|} \colon\, S \subseteq V(G),\, |S|=m\right\}\) because if \(k\) is the valency of \(G\), then \(k\cdot|S|=2\cdot|E(S)|+|\partial S|\) for every \(S\subseteq V(G)\).
    We claim that the isoperimetric number $$i(G)=\min_{m\leq\frac12q_1\cdots q_n} i_m(G)$$ is obtained when \(m=m_0:=\lfloor\frac{q_1}2\rfloor\cdot q_2\cdots q_n\), in which case it is equal to \(\lceil\frac{q_1}2\rceil\), the value in the statement. Note that every vertex in \(S_{m_0}\) has \(\lceil\frac{q_1}2\rceil\) neighbours outside \(S_{m_0}\).

    Choose an arbitrary integer \(m\) with \(m\leq\frac12q_1\cdots q_n\). If \(m\leq m_0\), then every vertex in \(S_m\) has at least \(\lceil\frac{q_1}2\rceil\) neighbours outside \(S_m\), which means that \(i_m(G)=\frac{|\partial S_m|}{|S_m|}\geq\lceil\frac{q_1}2\rceil\).
   
    If \(m>m_0\), then \(q_1\) is necessarily odd. If we define \(T=S_m\setminus S_{m_0}\), then \(T\) is an initial segment of the set of vertices of a ``slice'' \(H\cong K_{q_2}\times\dots\times K_{q_n}\), see Figure~\ref{fig:Hamming}.
    \begin{figure}[H]
        \centering
        \newcommand{\gridlengte}{1.3}
\begin{tikzpicture}
    \path (-3*\gridlengte,0) rectangle (4*\gridlengte,0);
    %H
    \draw[white, rounded corners, fill=black!20] (-1.8*\gridlengte,-.5) rectangle (2.5*\gridlengte,.5) node[black!30,xshift=-173,yshift=-14] {\(H\)};
    %white vertices
    \path[every node/.append style={circle, draw=black, thick, fill=white, minimum size=10pt, label distance=1pt, inner sep=0pt}]
    (2*\gridlengte,0) node (13) {}
    (2*\gridlengte,-\gridlengte) node (23) {}
    (\gridlengte,0) node (12) {}
    (-\gridlengte,-\gridlengte) node (20) {}
    (0,-\gridlengte) node (21) {}
    (\gridlengte,-\gridlengte) node (22) {};
    %vertices of Sm0
    \path[every node/.append style={circle, draw=black, thick, fill=black!50, minimum size=10pt, label distance=2pt, inner sep=0pt}]
    (-\gridlengte,\gridlengte) node[label=180:\color{black!70}\(S_{m_0}\)] (00) {}
    (0,\gridlengte) node (01) {}
    (\gridlengte,\gridlengte) node (02) {}
    (2*\gridlengte,\gridlengte) node (03) {};
    %vertices of T
    \path[every node/.append style={circle, draw=black, thick, fill=black, minimum size=10pt, label distance=6pt, inner sep=0pt}]
    (-\gridlengte,0) node[label=180:{\(T\)}] (10) {}
    (0,0) node (11) {};
    %edges
    \draw[thick] (00) edge (01) edge (10)
    (01) edge (02) edge (11)
    (02) edge (03) edge (12)
    (03) edge (13)
    (10) edge (11) edge (20)
    (11) edge (12) edge (21)
    (12) edge (13) edge (22)
    (13) edge (23)
    (20) edge (21)
    (21) edge (22)
    (22) edge (23)
    (23);
    \draw[thick] (00) to[out=15,in=165] (02)
    (01) to[out=15,in=165] (03)
    (10) to[out=15,in=165] (12)
    (11) to[out=15,in=165] (13)
    (20) to[out=15,in=165] (22)
    (21) to[out=15,in=165] (23)
    (00) to[out=-75,in=75] (20)
    (01) to[out=-75,in=75] (21)
    (02) to[out=-75,in=75] (22)
    (03) to[out=-75,in=75] (23)
    (00) to[out=20,in=160] (03)
    (10) to[out=20,in=160] (13)
    (20) to[out=20,in=160] (23);
\end{tikzpicture}
        \caption{\(T=S_m\setminus S_{m_0}\) (black) is an initial segment of the set of vertices of the ``slice'' \(H\) (light grey).}
        \label{fig:Hamming}
    \end{figure}
    Every vertex of \(T\) has equally many neighbours in \(S_{m_0}\) as in \(V(G)\setminus(S_{m_0}\cup V(H))\), so $|\partial S_{m}|=|\partial S_{m_0}|+|\partial_H T|$.
    By the induction hypothesis, \(\frac{|\partial_HT|}{|T|}\geq\lceil\frac{q_2}2\rceil\geq\lceil\frac{q_1}2\rceil\), where \(\partial_HT\) denotes the border of \(T\) within \(H\). Hence
    \[i_m(G)=\frac{|\partial S_m|}{|S_m|}=\frac{|\partial S_{m_0}|+|\partial_HT|}{|S_{m_0}|+|T|}\geq\left\lceil\frac{q_1}2\right\rceil,\]
    which proves that \(i(G)\geq\lceil\frac{q_1}2\rceil\).
\end{proof}

In the same way, we can also show a similar statement for the Cartesian product of complete bipartite graphs:
\begin{proposition}
    \(i(K_{q_1,q_1}\cartesianproduct\cdots\cartesianproduct K_{q_n,q_n})=\min_i\lceil\frac{q_i^2}2\rceil/q_i\)
\end{proposition}
\begin{proof}
    Similar to the proof of Proposition~\ref{prop:Hamming}. For the induction basis, we refer to Table~\ref{tab:exactnumber}. The optimality of lexicographically ordered sets, which is needed for this argument, is shown in \cite{AhlswedeCai}.\end{proof}

%%%%%%%%%%%%%%%%%%%%%%%%%%%%%%%%%%%%%%%%%%%%%%%%%%%%%%%%%%%%%%%%%%%%%%%%%%%
\subsection{Johnson graphs}
%%%%%%%%%%%%%%%%%%%%%%%%%%%%%%%%%%%%%%%%%%%%%%%%%%%%%%%%%%%%%%%%%%%%%%%%%%%
Let \(k\leq n\). The \emph{Johnson graph $J(n,k)$} has as vertices the \(k\)-sets of $\{1,\dots,n\}$, where two vertices are adjacent if they intersect in a \((k-1)\)-set. For \(J(n,2)\), the so-called \emph{triangular graph}, the isoperimetric problem was studied in a slightly different context:

\begin{example}
    The optimal isoperimetric sets for \(J(n,2)\) have already been determined \cite{Ahl78}, but deriving from them an exact expression for the isoperimetric number seems to be non-trivial. It was shown in \cite{triangular} that \(i(J(n,2))\sim (2-\sqrt{2})n\) asymptotically. From the results in that paper, one can also conclude that if \(2n(n-1)+1\) is a square, then \(i(J(n,2))=\frac{2q(q+1)}{n+q}\) where \(q=\frac12(1+\sqrt{2n(n-1)+1})\).
\end{example}

%%%%%%%%%%%%%%%%%%%%%%%%%%%%%%%%%%%%%%%%%%%%%%%%%%%%%%%%%%%%%%%%%%%%%%%%%%%
\subsection{Tight sets}
%%%%%%%%%%%%%%%%%%%%%%%%%%%%%%%%%%%%%%%%%%%%%%%%%%%%%%%%%%%%%%%%%%%%%%%%%%%

\begin{definition}\label{def:tight}
Let $G$ be a regular graph on $n$ vertices. Following \cite{intriguing}, we define:
\begin{itemize}
    \item If $S$ is a set of vertices satisfying \(\frac{|\partial S|}{|S|}=\left(1-\frac{|S|}{n}\right)\mu_n\), then $S$ is called a \emph{tight set of Type I}. These sets are exactly the ones meeting the upper bound of Theorem~\ref{thm:interlacing}
    \item  If $S$ is a set of vertices satisfying  $\frac{|\partial S|}{|S|}=(1-\frac{|S|}{n})\mu_2$, then $S$ is called a \emph{tight set of Type II}. These sets are exactly the ones meeting the lower bound of Theorem~\ref{thm:interlacing}. 
\end{itemize}

Following \cite{Bamberg2007,Bamberg2009,StronglyBrouwer,Drudge,Payne}, we simply write \emph{tight set} without further specification, to refer to a tight set of Type II.
\end{definition}

Note that the notions of a tight set (that is, a tight set of Type II) and tight set of Type I are complementary to each other, as the next result illustrates.

\begin{lemma}\label{lem:tightcomplement}
    A tight set of Type I is a tight set in the complementary graph $G^c$.
\end{lemma}
\begin{proof}
    Since \(\mu_i(G^c)+\mu_{n+2-i}(G)=n\) and \(|\partial_{G^c}(S)|+|\partial_G(S)|=|S|\cdot\left(n-|S|\right)\), we have that
    \[\frac{|\partial_G(S)|}{|S|}=\left(1-\frac{|S|}{n}\right)\mu_2(G)\quad\Longleftrightarrow\quad\frac{|\partial_{G^c}(S)|}{|S|}=\left(1-\frac{|S|}{n}\right)\mu_n(G^c).\qedhere\]
\end{proof}

The notion of tight sets was first introduced in the context of generalized quadrangles in \cite{Payne}, then extended to polar spaces of higher rank in \cite{Drudge} and later to regular graphs in \cite{intriguing}. 

By definition of a tight set, we have the following.

\begin{corollary}\label{cor:tight}
    If \(G\) is a regular graph on $n$ vertices, then \(i(G) \geq\frac12\mu_2\) with equality if and only if there is a tight set of size $\frac{n}{2}$.  
\end{corollary}
\begin{proof}
    This follows from Theorem~\ref{thm:interlacing} since \(|S|\leq\frac{n}2\Leftrightarrow1-\frac{|S|}{n} \ge \frac 12\).
\end{proof}

Note that, if $S$ is a tight set of size $\frac{n}{2}$, then every vertex in $S$ has \(i(G)=\frac12\mu_2\) neighbours outside $S$ and \(k-\frac12\mu_2\) neighbours within $S$, where $k$ is the valency of $G$.

Corollary~\ref{cor:tight} is a powerful tool to determine the isoperimetric number in any graph class where tight sets are understood, as we illustrate in Sections \ref{sec:grassman} and \ref{subsec:polar}. As mentioned above, tight sets were originally introduced in a geometrical context, so it should not come as a surprise that it is exactly for those geometrical graph classes that we can say a lot about the existence of tight sets, and hence about graph expansion. Despite that, we are not aware of previous research studying the isoperimetric number of geometric graphs. The remainder of this section is dedicated to a number of cases in which the exact isoperimetric number can be determined.

%%%%%%%%%%%%%%%%%%%%%%%%%%%%%%%%%%%%%%%%%%%%%%%%%%%%%%%%%%%%%%%%%%%%%%%%%%%
\subsection{Grassmann graphs}\label{sec:grassman}
%%%%%%%%%%%%%%%%%%%%%%%%%%%%%%%%%%%%%%%%%%%%%%%%%%%%%%%%%%%%%%%%%%%%%%%%%%%

Let \(q\) be a prime power and \(k\leq n\). Let \(\mathbb{F}\) denote the finite field with \(q\) elements. The \emph{q-Johnson graph} or \emph{Grassmann graph $J_q(n,k)$} has as vertices the \(k\)-spaces (vector spaces of dimension \(k\)) of the vector space $\mathbb{F}_q^n$, where two vertices are adjacent if they intersect in a \((k-1)\)-space. This graph is distance regular with the smallest positive Laplacian eigenvalue equal to $\mu_2=\lambda_1-\lambda_2=\frac{q^n-1}{q-1}$, see for example \cite{boolean1}.

Let \(\mathrm{PG}(n-1,q)\) denote the \(n\)-dimensional projective space over \(\mathbb{F}_q\). That is, points of \(\mathrm{PG}(n-1,q)\) are vector lines of \(\mathbb{F}_q^n\), lines of \(\mathrm{PG}(n-1,q)\) are vector planes of \(\mathbb{F}_q^n\), and so on.
A \emph{Cameron-Liebler line class} in \(\mathrm{PG}(n-1,q)\) is a set of lines whose characteristic vector is contained in the rowspan (over \(\mathbb{R}\)) of the point-line incidence matrix of \(\mathrm{PG}(n-1,q)\).
Drudge \cite{Drudge} showed that a tight set in \(J_q(4,2)\) is the same as a Cameron-Liebler line class in PG\((3,q)\).

Let \(\mathcal{L}\) denote the set of lines in PG\((3,q)\), so \(|\mathcal{L}|=(q^2+1)(q^2+q+1)\). We use the following result on Cameron-Lieber line classes:

\begin{theorem}[\cite{CLbruen}]\label{thm:CLbruen}
    There is a Cameron-Liebler line class of 
    size \(\frac12|\mathcal{L}|\) in \(\mathrm{PG}(3,q)\) if \(q\) is odd.
\end{theorem}

\begin{corollary}
    \(i(J_q(4,2))=\frac12(q^2+1)(q+1)\) if \(q\) is odd.
\end{corollary}
\begin{proof}
    Let $q$ be odd. By Theorem~\ref{thm:CLbruen}, $J_q(4,2)$ has a tight set of size $\frac12|\mathcal{L}| = \frac12|V(J_q(4,2))|$. Thus, we can apply Corollary~\ref{cor:tight}, using $\mu_2 = \frac{q^4-1}{q-1} = (q^2+1)(q+1)$.
\end{proof}

If we want to use the same approach for the Grassmann graph $J_q(n,k)$ with \(k\geq3\), one needs a so-called \emph{Cameron-Liebler set of \(k\)-spaces} \cite{CLk} consisting of half the number of $k$-spaces in PG$(n-1,q)$.
However, no such construction seems to be known for $k \geq 3$. In fact, in some cases it is known that no Cameron-Liebler set of $k$-spaces exists \cite{%BeuleCLk,JDB25,CLk,
boolean1}.
Note that a tight set in $J_q(n,k)$ of size $\frac12|V(J_q(n,k))|$ cannot exist if $q$ is even for divisibility reasons: if $q$ is even, then the number of vertices of $J_q(n,k)$ is odd.
Furthermore, it is known that for any fixed value of $q$, there are only a finite number of values of $n$ and $k$ for which there exist Cameron-Liebler sets consisting of precisely half the number of $k$-spaces in PG$(n,q)$. This is implied by the complete classification given in \cite{Ihringer}.

%%%%%%%%%%%%%%%%%%%%%%%%%%%%%%%%%%%%%%%%%%%%%%%%%%%%%%%%%%%%%%%%%%%%%%%%%%%
\subsection{Polar graphs}\label{subsec:polar}
%%%%%%%%%%%%%%%%%%%%%%%%%%%%%%%%%%%%%%%%%%%%%%%%%%%%%%%%%%%%%%%%%%%%%%%%%%%

Let $\mathcal{P}$ be a classical polar space over $\mathbb{F}_q$. We refer to \cite[Chapter 2]{StronglyBrouwer} for the definitions and basic properties of polar spaces, including spectral data of their (strongly regular) collinearity graphs.

A maximal singular subspace with respect to inclusion is called a \emph{generator}. 
Every generator has the same vector space dimension $r$, which is called the \emph{rank} of the polar space. In particular, any generator has $\frac{q^r-1}{q-1}$ points. Every singular subspace of dimension \(r-1\) is contained in the same number of generators. This number is denoted by \(t\). It holds that \(t \in \{1,q^\frac12,q,q^\frac32,q^2\}\). In particular, \(\mathcal{P}\) has \((tq^{r-1}+1)\cdot\frac{q^r-1}{q-1}\) points. %The tuple $(q,t)$ is called the \emph{order} of $\mathcal{P}$.

In this section, we consider the collinearity graph $\Gamma(\mathcal{P})$ on the points of a polar space, where two points are adjacent if they are in a common generator. 

To obtain tight sets of the desired size in the collinearity graphs, we use the concept of spreads.
A \emph{partial spread} (of generators) is a set of generators that pairwise intersect trivially. A \emph{spread} (of generators) is a partial spread that covers all points of $\mathcal{P}$. 

\begin{lemma}\label{lem:partialspread}
     If there exists a partial spread 
     covering exactly half the number of points, then the isoperimetric number of the collinearity graph \(\Gamma(\mathcal{P})\) is equal to \(\frac{1}{2}(tq^{r-1}+1)\frac{q^{r-1}-1}{q-1}\).
\end{lemma}
\begin{proof}
    A generator is a tight set, and the disjoint union of tight sets is again a tight set, see \cite[Theorem~8.1]{Drudge} for a characterization of tight sets in polar spaces.
    Hence, partial spreads are tight.
     
      We conclude from Corollary~\ref{cor:tight} that $i(G) = \frac12\mu_2$, which can be computed to be equal to \(\frac12(tq^{r-1}+1)\frac{q^{r-1}-1}{q-1}\) using the spectral data in \cite[Theorem 2.2.12]{StronglyBrouwer}.     
\end{proof}

If there is a (full) spread with an even number of generators (in which case $q$ must be odd), the condition in Lemma~\ref{lem:partialspread} is met by taking half of a spread.
In fact, in the following examples, the polar space always has a spread.  
We refer to \cite[Table~7.4]{HT2016} for an overview of existence results on spreads in polar spaces. 

\begin{itemize}
    \item The collinearity graph of the hyperbolic quadric \(Q^+(3,q)\) ($t=1$) has isoperimetric number \(i(\Gamma(Q^+(3,q)))=\frac{q+1}{2}\) if \(q\) is odd, by taking half the number of lines of a regulus (line spread). Note that this graph is in fact isomorphic to the square lattice \(K_{q+1}\cartesianproduct K_{q+1}\cong H(2,q+1)\), so this is a specific instance of the results on Hamming graphs (see Section \ref{sec:hamming}).
    \item For the elliptic quadric \(Q^-(5,q)\) ($t=q^2$) we have \(i(\Gamma(Q^-(5,q)))=\frac{q^3+1}{2}\) if \(q\) is odd, since $Q^-(5,q)$ has a spread of $q^3+1$ lines.
    \item The parabolic quadric \(Q(6,q)\) ($t=q$) when $q$ is odd and either prime or $q \equiv 0,2 \mod 3$, has a spread of $q^3+1$ generators, so its isoperimetric number is $i(\Gamma(Q(6,q)) = \frac12(q^3+1)(q+1)$.
    \item  The hyperbolic quadric $Q^+(7,q)$ ($t=1$), in the case that $q$ is odd and either prime or $q \equiv 0,2 \mod 3$, has a spread of $q^3+1$ generators. So we have $i(\Gamma(Q^+(7,q)) = \frac 12 (q^3+1)(q^2+q+1)$.
    \item The symplectic polar space $W(2r-1,q)$ ($t=q$) has a spread of size $q^r+1$. Therefore, if $q$ is odd, we have $i(\Gamma(W(2r-1,q))) = \frac12(q^r+1)\frac{(q^{r-1}-1)}{q-1}$.
\end{itemize}

The following examples serve to show that there are also cases where one can find a tight set of size $\frac12 |\Gamma(\mathcal{P})|$ without large partial spreads existing:
\begin{example}\label{exa:Q+(5,q)}
    The maximum size of a partial spread in \(Q^+(5,q)\) is two. However, by the Klein correspondence, we have that $\Gamma(Q^+(5,q))\cong J_q(4,2)$ \cite{Drudge}. So if \(q\) is odd, there exists a tight set by Theorem~\ref{thm:CLbruen} and \(i(\Gamma(Q^+(5,q)))=\frac12(q^2+1)(q+1)\).
\end{example}
\begin{example}
    The parabolic quadric $Q(4,q)$, \(q\) odd, does not have spreads, but it does have a tight set of size half the number of vertices. It also arises from the same Cameron-Liebler line class as in Theorem~\ref{thm:CLbruen}, see \cite{Bamberg2009}. Hence $i(\Gamma(Q(4,q)) = \frac12(q^2+1)$.
\end{example}

The complement of the collinearity graph of a polar space is again strongly regular and worth investigating. Tight sets %of Type II \jim{I left this one on purpose, I think it is worthwhile to specify type II anyway if we talk about type I at the same time to avoid confusion}\robin{I would put it between brackets}\jim{Now that we don't say Type II in the Lemma, I vote for removing of Type II here as well, but I'll wait to see if we agree on Lemma 5.4 being fine}\robin{Ok I agree on not mentioning Type II here}
in these graphs are tight sets of Type I in the collinearity graphs, see~Lemma~\ref{lem:tightcomplement}.
Still, we can derive the isoperimetric number of some of these complements, because tight sets of Type I are equivalent to so-called \(m\)-ovoids \cite[Proposition~2.1]{DeBruyn2008}. 
An \emph{\(m\)-ovoid} is a set of points $\mathcal{O}$ such that every generator meets $\mathcal{O}$ in $m$ points.

It is known that an $m$-ovoid has $|\mathcal{O|} = m(tq^{r-1}+1)$ points \cite{StronglyBrouwer}, so an $m$-ovoid containing half of the points of the polar space has $m = \frac12\frac{q^r-1}{q-1}$. If such an $m$-ovoid \(\mathcal{O}\) exists, then a double counting of the collinear pairs \((p_1,p_2)\) with \(p_1\in\mathcal{O}\) and \(p_2\notin\mathcal{O}\), results in an isoperimetric number that is equal to \(\frac12t(q^2-1)\). 
We obtain the following examples:
\begin{itemize}
    \item  The hyperbolic quadric $Q^+(3,q)$ ($t=1$), whose collinearity graph is isomorphic to the Hamming graph $H(2,q+1)$, contains an $\frac{q+1}{2}$-ovoid whenever $q$ is odd. The construction is circulant; identifying the quadric with the grid \(\mathbb{Z}/(q+1)\mathbb{Z}\times\mathbb{Z}/(q+1)\mathbb{Z}\), we can set $$\mathcal{O} = \left\{(i,j)\in\mathbb{Z}/(q+1)\mathbb{Z}\times\mathbb{Z}/(q+1)\mathbb{Z}\colon\, i-j \in \{0,\dots,\frac{q+1}{2}-1\}\right\}$$
    which is indeed a $\frac{q+1}{2}$-ovoid. 
    Hence, we get $i(\overline{\Gamma(Q^+(3,q))})=i(\overline{H(2,q+1)}) = \frac{q^2-1}{2}$ if \(q\) is odd.
    \item For the parabolic quadric $Q(4,q)$ ($t=q$) when $q$ is odd, there exist $\frac{q+1}{2}$-ovoids, see \cite[Section 3]{Bamberg2009}. We get $i(\overline{\Gamma(Q(4,q))})=\frac{q^3-q}{2}$.
    \item Consider the elliptic quadric $Q^-(5,q)$ ($t=q^2$) for $q$ odd. There exist $\frac{q+1}{2}$-ovoids, see \cite{Coss2005}. Hence we get that $i(\overline{\Gamma(Q^-(5,q))}) = \frac{q^4-q^2}{2}$.
    \item The symplectic quadric $W(3,q)$ ($t=q$) where, again, $q$ is odd, is also known to have $\frac{q+1}{2}$-ovoids, see \cite{Coss2008}. This implies $i(\overline{\Gamma(W(3,q)}) = \frac{q^3-q}{2}$. 
\end{itemize}

The results of this subsection on polar graphs are summarized in Table~\ref{tab:exactnumbergeom}.

\begin{table}[H]
\renewcommand{\arraystretch}{1.5}
    \centering
    \begin{tabular}{l|l}
        \(G\) & \(i(G)\)% & Reference 
        \\
        \hline\hline
        \(J_q(4,2)\cong\Gamma(Q^+(5,q))\), \(q\) odd & \(\frac12(q^2+1)(q+1)\)\\% & \cite{CLbruen}\\
        \(\Gamma(Q^+(3,q))\) & \(\lceil\frac{q+1}2\rceil\)\\% & Proposition~\ref{prop:Hamming}\\
        \(\Gamma(Q^-(5,q))\), \(q\) odd & \(\frac12(q^3+1)\)\\% & Lemma and HT\\
        \(\Gamma(Q(6,q))\), \(q\) odd and ($q$ prime or $q \equiv 0,2 \mod 3$) & \(\frac12(q^3+1)(q+1)\)\\
        \(\Gamma(Q^+(7,q))\), \(q\) odd and ($q$ prime or $q \equiv 0,2 \mod 3$) & \(\frac12(q^6-1)/(q-1)\)\\
        \(\Gamma(W(2r-1,q))\), \(q\) odd & \(\frac12(q^r+1)(q^{r-1}-1)/(q-1)\)\\
        \(\Gamma(Q(4,q))\), \(q\) odd & \(\frac12(q^2+1)\)\\
        \(\overline{\Gamma(Q^+(3,q))}\), \(q\) odd & \(\frac12(q^2-1)\)\\
        \(\overline{\Gamma(Q(4,q))}\), \(q\) odd & \(\frac12q(q^2-1)\)\\
        \(\overline{\Gamma(Q^-(5,q))}\), \(q\) odd & \(\frac12q^2(q^2-1)\)\\
        \(\overline{\Gamma(W(3,q))}\), \(q\) odd & \(\frac12q(q^2-1)\)
    \end{tabular}
    \caption{The isoperimetric number \(i(G)\) of some families of graphs from geometries, attaining the bound in Corollary~\ref{cor:tight}.}
    \label{tab:exactnumbergeom}
\end{table}

%%%%%%%%%%%%%%%%%%%%%%%%%%%%%%%%%%%%%%%%%%%%%%%%%%%%%%%%%%%%%
\section{The isoperimetric number is not determined by the spectrum}\label{sec:NDS}
%%%%%%%%%%%%%%%%%%%%%%%%%%%%%%%%%%%%%%%%%%%%%%%%%%%%%%%%%%%%%

There are many graph parameters that can be bounded by eigenvalues. One can wonder whether such properties or parameters are actually determined by the spectrum. Typically these are not. Instances of such properties include: distance-regularity \cite{h1996}, having a given diameter \cite{h1996}, admitting a perfect matching \cite{bch2015}, Hamiltonicity \cite{lwyl2010}, vertex or edge-connectivity \cite{h2020} and having a given zero forcing number \cite{MRC}.
In this section, we add ``having a given isoperimetric number'' to that list.
In fact, we show that there is an infinite family of pairs of regular connected cospectral graphs that have different isoperimetric numbers. We note that this has already been proven for the vertex-isoperimetric number in \cite{VertexNotIso}. 

The hypercube \(Q_n\cong H(n,2)\) has isoperimetric number \(i(Q_n)=1\) \cite{M1989}. Indeed, any subcube \(Q_{n-1}\) is a tight set of size \(2^{n-1}\), see Corollary~\ref{cor:tight}. Consider the tesseract \(Q_4\). It was shown in \cite{Hoffman63} that $Q_4$ has exactly one cospectral mate, which we denote by \(\widetilde{Q_4}\). The graph \(\widetilde{Q_4}\) can be obtained from \(Q_4\) by GM-switching \cite{GM82} on the four neighbours of a fixed vertex.

\begin{theorem}[Godsil-McKay switching \cite{GM82} on four vertices]\label{thm:GM}
    Consider a graph \(G\) and a subset \(C\subseteq V(G)\) with \(|C|=4\) such that the induced subgraph \(G[C]\) is regular and every vertex \(x\notin C\) has exactly zero, two, or four neighbours in \(C\). For every \(x\notin C\) that has exactly two neighbours in \(C\), make it adjacent to the other two vertices of \(C\) instead. The obtained graph has the same adjacency and Laplacian spectrum as \(G\).
\end{theorem}

\begin{figure}
    \centering
    \newcommand{\gridlengte}{1.2}
\begin{tikzpicture}
    \path (-2.5*\gridlengte,0) rectangle (3.5*\gridlengte,0);
    %vertices
    \path[every node/.append style={circle, draw=black, thick, fill=black, minimum size=3pt, label distance=1pt, inner sep=0pt}]
    (0,0) node (0000) {}
    (-1.5*\gridlengte,\gridlengte) node (1000) {}
    (-.5*\gridlengte,\gridlengte) node (0100) {}
    (.5*\gridlengte,\gridlengte) node (0010) {}
    (2.5*\gridlengte,\gridlengte) node (0001) {}
    (-2*\gridlengte,2*\gridlengte) node[label=180:\scriptsize\(u_1\)] (1100) {}
    (-\gridlengte,2*\gridlengte) node[label=180:\scriptsize\(u_2\)] (1010) {}
    (\gridlengte,2*\gridlengte) node (1001) {}
    (0,2*\gridlengte) node[label=180:\scriptsize\(u_3\)] (0110) {}
    (2*\gridlengte,2*\gridlengte) node (0101) {}
    (3*\gridlengte,2*\gridlengte) node (0011) {}
    (-1.5*\gridlengte,3*\gridlengte) node[label=90:\(w\)] (1110) {}
    (.5*\gridlengte,3*\gridlengte) node (1101) {}
    (1.5*\gridlengte,3*\gridlengte) node (1011) {}
    (2.5*\gridlengte,3*\gridlengte) node (0111) {}
    (\gridlengte,4*\gridlengte) node[label=90:\(v\)] (1111) {};
    %edges
    \draw[thick] (0000) edge (1000) edge (0100) edge (0010) edge (0001)
    (1100) edge (1000) edge (0100) edge (1110) edge (1101)
    (1010) edge (1000) edge (0010) edge (1110) edge (1011)
    (1001) edge (1000) edge (0001) edge (1101) edge (1011)
    (0110) edge (0100) edge (0010) edge (1110) edge (0111)
    (0101) edge (0100) edge (0001) edge (1101) edge (0111)
    (0011) edge (0010) edge (0001) edge (1011) edge (0111)
    (1111) edge (1110) edge (1101) edge (1011) edge (0111);
\end{tikzpicture}
\begin{tikzpicture}
    \path (-2.5*\gridlengte,0) rectangle (3.5*\gridlengte,0);
    %vertices
    \path[every node/.append style={circle, draw=black, thick, fill=black, minimum size=3pt, label distance=1pt, inner sep=0pt}]
    (0,0) node (0000) {}
    (-1.5*\gridlengte,\gridlengte) node (1000) {}
    (-.5*\gridlengte,\gridlengte) node (0100) {}
    (.5*\gridlengte,\gridlengte) node (0010) {}
    (2.5*\gridlengte,\gridlengte) node (0001) {}
    (-2*\gridlengte,2*\gridlengte) node[label=180:\scriptsize\(u_1\)] (1100) {}
    (-\gridlengte,2*\gridlengte) node[label={[yshift=-1pt]180:\scriptsize\(u_2\)}] (1010) {}
    (\gridlengte,2*\gridlengte) node (1001) {}
    (0,2*\gridlengte) node[label={[yshift=-1pt]180:\scriptsize\(u_3\)}] (0110) {}
    (2*\gridlengte,2*\gridlengte) node (0101) {}
    (3*\gridlengte,2*\gridlengte) node (0011) {}
    (-1.5*\gridlengte,3*\gridlengte) node[label=90:\(w\)] (1110) {}
    (.5*\gridlengte,3*\gridlengte) node (1101) {}
    (1.5*\gridlengte,3*\gridlengte) node (1011) {}
    (2.5*\gridlengte,3*\gridlengte) node (0111) {}
    (\gridlengte,4*\gridlengte) node[label=90:\(v\)] (1111) {};
    %edges
    \draw[thick] (0000) edge (1000) edge (0100) edge (0010) edge (0001)
    (1100) edge (1000) edge (0100) edge (0111) edge (1011)
    (1010) edge (1000) edge (0010) edge (0111) edge (1101)
    (1001) edge (1000) edge (0001) edge (0111) edge (1110)
    (0110) edge (0100) edge (0010) edge (1011) edge (1101)
    (0101) edge (0100) edge (0001) edge (1011) edge (1110)
    (0011) edge (0010) edge (0001) edge (1101) edge (1110)
    (1111) edge (1110) edge (1101) edge (1011) edge (0111);
\end{tikzpicture}
    \caption{The tesseract \(Q_4\) and its cospectral mate \(\widetilde{Q_4}\). This figure is taken from \cite[Figure 1.3]{spectra}, with adapted labelling.}
    \label{fig:placeholder}
\end{figure}

\begin{proposition}
The infinite family of pairs of regular connected graphs
\[Q_4\cartesianproduct H(n,4)\quad\text{ and }\quad\widetilde{Q_4}\cartesianproduct H(n,4),\]
have the same adjacency and Laplacian spectrum but different isoperimetric numbers.
\end{proposition}

\begin{proof}
    The pairs of graphs in the statement are cospectral because the spectrum of a product is determined by that of its factors. The graphs are regular, so they are also cospectral for the Laplacian spectrum. We have that $i(Q_4)=1$, and it is known that $i(G \cartesianproduct K_{2n}) = \min \{i(G),n\}$ for all \(n\geq2\), see \cite[Theorem~5.1]{M1989}. In particular, $i(Q_4 \cartesianproduct H(n,4)) = 1$ and $i(\widetilde{Q_4}\cartesianproduct H(n,4)) = \min\{i(\widetilde{Q_4}),2\}$. We prove that \(i(\widetilde{Q_4})>1\), and hence \(i(Q_4 \cartesianproduct H(n,4))\neq i(\widetilde{Q_4}\cartesianproduct H(n,4))\). One can check by computer that \(i(\widetilde{Q_4})=5/4\), but here we provide a purely mathematical proof.

By Corollary~\ref{cor:tight}, it suffices to prove that \(\widetilde{Q_4}\) does not have a tight set of size eight. Suppose, by contradiction, that \(S\) is a tight set of size eight in \(\widetilde{Q_4}\). Every vertex in \(S\) has three neighbours in \(S\) and one neighbour in \(S^c\) (and vice versa).
Let \(v\) be the chosen vertex such that \(\widetilde{Q_4}\) is obtained by GM-switching on \(Q_4\) with respect to the four neighbours of \(v\), see Theorem~\ref{thm:GM} and the sentence above it. We may assume that \(v\in S\), otherwise we can consider \(S^c\) instead of \(S\). Now \(v\) has exactly one neighbour \(w\) with \(w\notin S\). Note that, in the original graph \(Q_4\), \(w\) is contained in a unique cube that does not contain \(v\). Let \(u_1,u_2\) and \(u_3\) be the three neighbours of \(w\) in that cube. They were originally adjacent to two neighbours of \(v\) (including \(w\)), so after switching they are adjacent to two neighbours of \(v\) in \(S\). Hence they are also contained in \(S\). Similarly, the vertices that used to be at distance two from \(w\) in the cube in \(Q_4\) are (still) adjacent to two of the vertices of \(\{u_1,u_2,u_3\}\), so they are also in \(S\). All the considered vertices are different, so we proved that \(|S|\geq10\), a contradiction.
\end{proof}

%%%%%%%%%%%%%%%%%%%%%%%%%%%%%%%%%%%%%%%%%%%%%%%%%%%%%%%%%%%%%%%%%%%%%%%%%%%
\section{The isoperimetric number of random split graphs}\label{sec:split}
%%%%%%%%%%%%%%%%%%%%%%%%%%%%%%%%%%%%%%%%%%%%%%%%%%%%%%%%%%%%%%%%%%%%%%%%%%%

In this section, we write $o(1)$ for quantities that tend to $0$ as $n \to \infty$ and say that an event holds with high probability (w.h.p.) if its probability tends to $1$ as $n \to \infty$.

Almost all $k$-regular graphs have isoperimetric number of order $k/2 - \Theta(\sqrt{n})$~\cite{Bollobas1988, Alon1986}. For irregular graphs, the situation seems to change. The simplest upper bound for the isoperimetric number is achieved by taking a vertex of minimum degree, and there is some evidence that this upper bound is typically tight for irregular graphs. In particular, \cite{BHKL2008} studied the isoperimetric number during a specific random graph process. In this process they start with the empty graph $E_n$, and at each step an edge is added uniformly at random from all the missing edges. The authors of \cite{BHKL2008} show that for almost all graphs in the process, $i(G) = \delta(G)$ as long as $\delta(G) = o(\log n)$. In this section we study a different class of irregular graphs and show that the random split graph satisfies $i(G) = \delta(G)$ with high probability. 

A graph $G = (V,E)$ is called \emph{split} if the vertex set can be partitioned into two disjoint subsets $V_1$ and $V_2$ such that $G[V_1] \cong K_k$ and $G[V_2] \cong E_\ell$ for some \(k\) and \(\ell\) \cite[Definition~3.2.3]{BLS1999}. Split graphs have attracted considerable attention within the spectral graph theory community in recent years, frequently appearing as extremal examples in eigenvalue maximization problems~\cite{byrne2024}, and playing a key role in the resolution of the Grone–Merris conjecture~\cite{bai2011}. Based on the work of Grone and Merris, in~\cite{AFHP2014} the authors showed a spectral upper bound for $i(G)$. As a byproduct of our results, we prove that this bound is tight w.h.p. for random split graphs.

We use the construction of \emph{random split graphs} from~\cite{BGM2017}. Given two positive integers $k$ and $\ell$, define $G_{k, \ell}$ as the following random graph. The vertex set of $G_{k, \ell}$ is a disjoint union \(V_1\cup V_2\), where \(V_1\) and \(V_2\) have size $k$ and $\ell$ respectively. The induced subgraph on \(V_1\) is complete and the induced subgraph on \(V_2\) is edgeless. Every pair \(\{v_1,v_2\}\) with \(v_1\in V_1\) and \(v_2\in V_2\), is an edge with probability \(1/2\). These $k\ell$ choices are made independently. The main result of this section is the following:

\begin{theorem}~\label{thm:iG_random_split_graphs}
    If $G_{k, k}$ is a random split graph, then with high probability
    \begin{equation*}
        i(G_{k, k}) = \delta(G_{k, k}).
    \end{equation*}
\end{theorem}

The proof idea is as follows. We consider two cases for the set $S$ that minimizes $i(G_{k, k})$. If $|S|\leq k/2$, we show that w.h.p. $G_{k,k}$ satisfies $\abs{\partial S}/\abs{S}\geq \delta(G_{k,k})$ (Lemma~\ref{lem:split_small_s}), using that it is a split graph with minimum degree at most $k/2$ (Lemma~\ref{lem:split_min_deg}). If $|S|\geq k/2$, we apply Chernoff's inequality (Lemma~\ref{lem:Chernoff}) together with a union bound to show that w.h.p. $\abs{\partial S}/\abs{S}$ is greater than $\delta(G_{k, k})$ (Lemma~\ref{lem:split_large_s}).

\begin{lemma}\label{lem:split_min_deg}
    With high probability
    \begin{equation*}
        \delta(G_{k, k}) \leq \frac{k}{2} - \frac{1}{2}\sqrt{k\log k}.
    \end{equation*}
\end{lemma}
\begin{proof}
    The degree of each vertex $v_i \in V_2$ is a binomial random variable $\text{deg}(v_i) \sim \text{Bin}(k,1/2)$. It follows from~\cite[Theorem 3.3]{Bollobas2001} that, w.h.p.,
    \begin{equation*}
        \min_{i} \text{deg}(v_i) = \frac{k}{2} - (1+o(1)) \sqrt{\frac{k\log k}{2}}.
    \end{equation*}
    Since $\delta(G_{k, k}) \leq \min_{i} \text{deg}(v_i)$, the result follows.
\end{proof}

\begin{lemma}\label{lem:split_small_s}
    With high probability, for every $S\subseteq V(G_{k,k})$ with $1 \leq |S| \leq k/2$,
    \begin{equation*}
        \frac{|\partial S|}{|S|} \geq \delta(G_{k,k}).
    \end{equation*}
\end{lemma}

\begin{proof}
    Let $s := |S|$ and $a := |S\cap V_1|$. Every vertex in $S\cap V_1$ has at least $k-a$ edges from $S$ to $V_1\setminus S$ and every vertex in $S\cap V_2$ has at least $\delta(G_{k,k})-a$ edges from $S$ to $V_1\setminus S$. So
    \begin{align*}
         \frac{|\partial S|}{|S|} &\geq \frac{(k-a)a + (\delta(G_{k,k}) - a)(s-a)}{s}\\
         &= \delta(G_{k,k}) + \frac{a}{s}(k - s - \delta(G_{k,k}))\\
         &\geq \delta(G_{k,k}).
    \end{align*}
    In the last inequality, we apply 
    Lemma~\ref{lem:split_min_deg}, which implies that w.h.p. $\delta(G_{k, k}) \leq k/2$.
\end{proof}

We proceed by using the typical Chernoff-type tail bound on the binomial distribution (see for example \cite[Appendix A]{AS2016}).

\begin{lemma}[{\cite{AS2016}}]\label{lem:Chernoff}
Let $n \in \mathbb{N}$, let $p \in [0,1]$, let $X \sim \text{Bin}(k,p)$ and let $\mu = kp$. For every $t$ with $0\le t \leq \mu/2$, 
\[
\mathbb{P}\left[X < \mu - t\right] \leq \exp\left(-\frac{t^2}{2\mu} \right).
\]
\end{lemma}

\begin{lemma}\label{lem:split_large_s}
    With high probability, for every $S\subseteq V(G_{k,k})$ with $k/2 \leq |S| \leq k$,
    \begin{equation*}
        \frac{|\partial S|}{|S|} \geq \frac{k}{2} - 2\sqrt{k}.
    \end{equation*}
\end{lemma}
\begin{proof}
    Fix a set $S$, with $S_1 = S \cap V_1$, $S_2 = S \cap V_2$, $\abs{S_1} = a, \abs{S_2} = b$. Define a random variable $X$ for the number of non-deterministic edges in $\partial S$, that is, $X\sim \text{Bin}(m,\frac{1}{2})$, where $m:= a(k-b)+b(k-a)$. We have
    \begin{equation*}
        \abs{\partial S} = a(k-a) + X.
    \end{equation*}
    If $B_S$ denotes the event that $\abs{\partial S} < \abs{S}k/2 - 2|S|\sqrt{k}$, then
    \begin{equation*}
        \Prob{B_S} = \Prob{X < \frac{\abs{S}k}{2} -a(k-a)- 2|S|\sqrt{k}}.
    \end{equation*}
    Let $\mu:=\Exp{X} = m/2$. Since $a+b \leq k$,
    \begin{equation*}
        \mu \geq \mu - a(k-a-b) = \frac{\abs{S}k}{2} -a(k-a).
    \end{equation*}
    Using this observation, we can apply Lemma~\ref{lem:Chernoff} to bound $\Prob{B_S}$, which gives
    \begin{equation*}
        \Prob{B_S} \leq \Prob{X < \mu - 2|S|\sqrt{k}}
        \leq \exp\left(-\frac{4|S|^2k}{2\mu}\right)
        \leq \exp\left(-4|S|\right),
    \end{equation*}    
    where the last inequality follows from $\mu \leq \abs{S}k/2$. The argument finishes by taking the union bound over all choices of $S$: 
    \begin{equation*}
        \Prob{\bigcup_{\substack{S \subseteq V(G_{k,k}),\\ k/2 \leq |S| \leq k}}B_S} \leq \sum_{\substack{S \subseteq V(G_{k,k}),\\ k/2 \leq |S| \leq k}}\Prob{B_S} \leq \sum_{s=k/2}^{k} \binom{2k}{s} e^{-4s} \leq 2^{2k}e^{-2k} \to 0.\qedhere
    \end{equation*}
\end{proof}

We are now ready to show the main result of this section.

\begin{proof}[Proof of Theorem~\ref{thm:iG_random_split_graphs}]
    Let $S \subseteq V(G_{k, k})$ with $\abs{\partial S}/\abs{S}=i(G_{k, k})$. From Lemma~\ref{lem:split_min_deg}, $\delta(G_{k, k}) \leq k/2$ w.h.p., so if $|S| \leq k/2$, we can apply Lemma~\ref{lem:split_small_s} and get that $\abs{\partial S}/\abs{S} = \delta(G_{k, k})$. Now, in the case $|S| \geq k/2$, we apply Lemma~\ref{lem:split_large_s} and w.h.p., for any choice of $S$,
    \begin{equation*}
        \frac{|\partial S|}{|S|} \geq \frac{k}{2} - 2\sqrt{k} \geq \delta(G_{k, k}), 
    \end{equation*}
    where the last inequality follows from Lemma~\ref{lem:split_min_deg}. 
\end{proof}

In \cite{AFHP2014}, the authors provide an upper bound on $i(G)$ based on the work of Grone and Merris on the majorization relation between the Laplacian eigenvalues and the degree sequence.

\begin{proposition}[{\cite[Proposition 4.6]{AFHP2014}}]
    \label{propo:Aida7b}
    \begin{equation}
    \label{eq:Aida7b}
        i(G)\leq  \min_{\frac{n}{2}\leq m < n} \sum_{i=1}^{m}(\mu_{n+1-i} - d_{i}).
    \end{equation}
\end{proposition}

The bound~\eqref{eq:Aida7b} beats the upper bound~\eqref{eq:Mohar} for the complete split graph $S_{p, q}$ obtained by adding all edges between $K_p$ and $E_q$~\cite[Example 4.7]{AFHP2014}. In this direction, Theorem~\ref{thm:iG_random_split_graphs} gives a family of split graphs for which equality occurs in~\eqref{eq:Aida7b}.

\begin{corollary}
   With high probability
    \begin{equation*}
        i\left(G_{k, k}\right) =  \min_{k\leq m < 2k} \sum_{i=1}^{m}(\mu_{n+1-i} - d_{i}).
    \end{equation*}
\end{corollary}
\begin{proof}
It follows from Theorem~\ref{thm:iG_random_split_graphs} that w.h.p. the minimum degree of $G_{k, k}$ is equal to the isoperimetric number. Moreover, the following holds for every graph $G$ on $n$ vertices:
\begin{align*}
    \min_{\frac{n}{2}\leq m < n} \sum_{i=1}^{m}(\mu_{n+1-i} - d_{i})& \leq \sum_{i=1}^{n-1}(\mu_{n+1-i} - d_{i})\\
    &= (2|E(G)|-\mu_1)-(2|E(G)|-d_n) \\&= \delta(G). \qedhere
\end{align*}
\end{proof}

%%%%%%%%%%%%%%%%%%%%%%%%%%%%%%%%%%%%%%%%%%%%%%%%%%%%%%%%%%%%%%%%%%
\section{Concluding remarks}
%%%%%%%%%%%%%%%%%%%%%%%%%%%%%%%%%%%%%%%%%%%%%%%%%%%%%%%%%%%%%%%%%%
Our results strengthen and extend several existing results on edge expansion. For graph powers, we extend classical spectral bounds to arbitrary powers, including cases where the spectra of $G$ and $G^t$ are not polynomially related; unlike previous results for weighted powers, our bounds apply to unweighted graphs and use only spectral and local information of $G$. For $t=2$, we determine optimal polynomial bounds that can strictly improve the classical bound applied directly to $G^2$, while for general $t$ we provide LP formulations for optimizing them. For distance-regular graphs, we extend an LP formula previously available only for strongly regular graphs, obtaining a bound determined only by the intersection array. We also show that the isoperimetric number is not determined by the spectrum and determine it for random split graphs. Finally, using a connection with tight sets, we turn previously known bounds for several graph families coming from finite geometries into exact isoperimetric values.

%%%%%%%%%%%%%%%%%%%%%%%%%%%%%%%%%%%%%%%%%%%%%%%%%%%%%%%%%%%%%%%%%%
\subsection*{Acknowledgments}
%%%%%%%%%%%%%%%%%%%%%%%%%%%%%%%%%%%%%%%%%%%%%%%%%%%%%%%%%%%%%%%%%%

Aida Abiad, Nils van de Berg and Harper Reijnders are supported by NWO (Dutch Research Council) through the grant VI.Vidi.213.085. Emanuel Juliano is supported by FAPEMIG and CNPq grants. Robin Simoens is supported by the Research Foundation Flanders (FWO) through the grant 11PG724N. Thijs van Veluw is supported by the Special Research Fund of Ghent University through the grant BOF/24J/2023/047, and by the Fund Professor Frans Wuytack. This research was partially supported by the FWO Scientific Research Communities: Graphs, Association schemes and Geometries: structures, algorithms and
computation (W003324N) and Finite Geometry, Coding Theory and Cryptography (WO.032.24N).
The authors thank Antonina Khramova, Leo Storme and Vladislav Taranchuk for their initial help with the research for this paper, and Ferdinand Ihringer for pointing out several references related to Section \ref{sec:finitegeometry}.

%%%%%%%%%%%%%%%%%%%%%%%%%%%%%%%%%%%%%%%%%%%%%%%%%%%%%%%%%%%%%%%

\newpage
\appendix
\renewcommand{\thetable}{A\arabic{table}}
\setcounter{table}{0}

%%%%%%%%%%%%%%%%%%%%%%%%%%%%%%%%%%%%%%%%%%%%%%%%%%%%%%%%%%%%%%%
\section{Appendix}\label{sec:app_sims}
%%%%%%%%%%%%%%%%%%%%%%%%%%%%%%%%%%%%%%%%%%%%%%%%%%%%%%%%%%%%%%%

\begin{table}[H]
    \centering
    \begin{tabular}{l|c|ccc|c|c|c}
    Graph & DRG & \multicolumn{4}{c|}{Theorem \ref{theorem:lowerboundi(G^2)closed}} & Theorem \ref{theorem:upperboundi(G^2)closed} & $i(G^2)$ \\
     &  & i. &ii. &  iii. & Bound &  &  \\
    \hline
    Bidiakis cube &  & x &  &  & $1.75$ & $6.0$ & $3.33$ \\
    Blanusa First Snark Graph &  & x &  &  & $1.98$ & $6.0$ & $2.5$ \\
    Blanusa Second Snark Graph &  & x &  &  & $1.86$ & $6.12$ & $2.33$ \\
    Brinkmann graph &  & x &  &  & $7.0$ & $10.31$ & $7.6$ \\
    Coxeter Graph & x & x &  &  & $3.0$ & $6.0$ & $3.29$ \\
    \textbf{Desargues Graph} & x & x & x & x & \textbf{3.0} & $6.0$ & $3.0$ \\
    Dodecahedron & x & x &  &  & $2.38$ & $6.0$ & $3.0$ \\
    Double star snark &  & x &  &  & $1.44$ & $6.12$ & $1.67$ \\
    Durer graph &  & x &  & x & $2.0$ & $6.0$ & $3.0$ \\
    Dyck graph &  & x &  &  & $2.38$ & $6.0$ & $2.5$ \\
    F26A Graph &  & x &  &  & $2.81$ & $6.07$ & $3.0$ \\
    Flower Snark &  & x &  &  & $2.83$ & $6.12$ & $3.4$ \\
    Folkman Graph &  & x &  &  & $2.03$ & $10.0$ & $4.0$ \\
    Franklin graph &  & x &  &  & $2.13$ & $6.0$ & $3.0$ \\
    Frucht graph &  & x &  &  & $1.14$ & $6.12$ & $2.67$ \\
    \textbf{Heawood graph} & x &  & x & x & \textbf{3.0} & $5.71$ & $3.0$ \\
    \textbf{Hexahedron} & x & x & x & x & \textbf{3.0} & $4.0$ & $3.0$ \\
    Hoffman Graph &  & x &  &  & $3.0$ & $10.0$ & $4.0$ \\
    Holt graph &  & x &  &  & $6.31$ & $10.13$ & $7.08$ \\
    \textbf{Icosahedron} & x & x & x &  & \textbf{5.0} & $6.0$ & $5.0$ \\
    Klein 7-regular Graph & x & x & x &  & $10.5$ & $12.0$ & $10.67$ \\
    Markstroem Graph &  & x &  &  & $0.96$ & $6.1$ & $1.83$ \\
    McGee graph &  & x &  &  & $3.0$ & $6.12$ & $3.33$ \\
    \textbf{Moebius-Kantor Graph} &  &  & x & x & \textbf{3.0} & $6.0$ & $3.0$ \\
    \textbf{Nauru Graph} &  & x & x & x & \textbf{3.0} & $6.0$ & $3.0$ \\
    \textbf{Pappus Graph} & x &  & x & x & \textbf{3.0} & $6.0$ & $3.0$ \\
    Robertson Graph &  & x & x &  & $7.5$ & $10.0$ & $8.0$ \\
    Sylvester Graph & x & x & x &  & $12.0$ & $15.0$ & time \\
    Tietze Graph &  &  &  & x & $2.04$ & $6.0$ & $4.0$ \\
    Tricorn Graph &  &  & x & x & $2.5$ & $6.11$ & $3.4$ \\
    \textbf{Tutte-Coxeter graph} & x & x & x & x & \textbf{3.0} & $6.0$ & $3.0$ \\
    \textbf{Twinplex Graph} &  & x & x &  & \textbf{4.0} & $6.0$ & $4.0$ \\
    Wells graph & x & x &  &  & $11.38$ & $14.0$ & $11.5$ \\
    \end{tabular}
    \caption{Comparison of the bounds from Theorem \ref{theorem:lowerboundi(G^2)closed} and Theorem \ref{theorem:upperboundi(G^2)closed} for Sage named graphs of diameter at least 3. Bounds in bold are tight.}
    \label{tab:t2_sim}
\end{table}

\begin{table}[H]
    \centering
       \begin{tabular}{l|c|cc|c}
    Graph & DRG & LP \eqref{eq:iso_LP_lower} & LP \eqref{eq:iso_LP_upper} & $i(G^3)$ \\
    \hline
    Balaban 10-cage &  & $6.55$ & $15.0$ & time \\
    Blanusa First Snark Graph &  & $2.92$ & $13.0$ & $6.33$ \\
    Blanusa Second Snark Graph &  & $2.77$ & $13.02$ & $6.44$ \\
    Bucky Ball &  & $1.68$ & $13.0$ & time \\
    Conway-Smith graph for 3S7 & x & $30.0$ & $32.0$ & time \\
    \textbf{Coxeter Graph} & x & \textbf{10.0} & $12.71$ & $10.0$ \\
    Desargues Graph & x & $6.5$ & $9.0$ & $6.6$ \\
    Dodecahedron & x & $6.38$ & $9.0$ & $6.6$ \\
    Double star snark &  & $3.44$ & $13.01$ & $5.67$ \\
    Durer graph &  & $3.38$ & $9.0$ & $5.5$ \\
    Dyck graph &  & $5.38$ & $15.0$ & $7.0$ \\
    Ellingham-Horton 54-graph &  & $0.82$ & $15.0$ & time \\
    Ellingham-Horton 78-graph &  & $0.26$ & $15.0$ & time \\
    F26A Graph &  & $6.17$ & $15.0$ & $8.0$ \\
    Flower Snark &  & $4.29$ & $12.93$ & $8.0$ \\
    Folkman Graph &  & $6.25$ & $13.0$ & $8.8$ \\
    Foster Graph & x & $6.55$ & $15.0$ & time \\
    Foster graph for 3.Sym(6) graph & x & $21.0$ & $23.0$ & time \\
    Frucht graph &  & $3.04$ & $12.83$ & $5.67$ \\
    Gray graph &  & $6.55$ & $15.0$ & time \\
    Harborth Graph &  & $0.87$ & $30.03$ & time \\
    Harries Graph &  & $6.55$ & $15.0$ & time \\
    Harries-Wong graph &  & $6.55$ & $15.0$ & time \\
    Hoffman Graph &  & $6.0$ & $12.0$ & $7.5$ \\
    Horton Graph &  & $0.41$ & $15.0$ & time \\
    Klein 3-regular Graph &  & $6.88$ & $12.41$ & time \\
    Markstroem Graph &  & $1.78$ & $13.03$ & $4.33$ \\
    \textbf{McGee graph} &  & \textbf{10.0} & $12.0$ & $10.0$ \\
    Meredith Graph &  & $0.85$ & $38.23$ & time \\
    Moebius-Kantor Graph &  & $6.21$ & $9.0$ & $7.0$ \\
    Nauru Graph &  & $6.5$ & $15.0$ & $8.67$ \\
    Pappus Graph & x & $7.5$ & $9.0$ & $7.67$ \\
    Szekeres Snark Graph &  & $1.41$ & $13.0$ & time \\
    Tutte-Coxeter graph & x & $10.0$ & $15.0$ & $10.2$ \\
    \textbf{Wells graph} & x & \textbf{15.0} & $16.0$ & $15.0$ \\
    \end{tabular}
    \caption{Comparison of the bounds from Theorem \ref{theorem:lowerboundi(G^t)} (LP \eqref{eq:iso_LP_lower}) and Theorem \ref{theorem:upperboundi(G^t)} (LP \eqref{eq:iso_LP_upper}) for $t=3$ for Sage named graphs of diameter at least 4. Bounds in bold are tight.}
    \label{tab:t3_sim}
\end{table}

\begin{table}[H]
    \centering
    \begin{tabular}{l|c|cc|c}
    Graph & DRG & LP \eqref{eq:iso_LP_lower}& LP \eqref{eq:iso_LP_upper} & $i(G^4)$ \\
    \hline
    Balaban 10-cage &  & $15.0$ & $26.41$ & time \\
    Bucky Ball &  & $4.98$ & $23.28$ & time \\
    \textbf{Desargues Graph} & x & \textbf{9.0} & $10.0$ & $9.0$ \\
    \textbf{Dodecahedron} & x & \textbf{9.0} & $10.0$ & $9.0$ \\
    Dyck graph &  & $9.0$ & $17.0$ & $12.0$ \\
    Ellingham-Horton 54-graph &  & $1.24$ & $26.55$ & time \\
    Ellingham-Horton 78-graph &  & $0.4$ & $26.56$ & time \\
    F26A Graph &  & $9.0$ & $15.38$ & $12.0$ \\
    Foster Graph & x & $15.0$ & $25.0$ & time \\
    Gray graph &  & $10.55$ & $26.37$ & time \\
    Harborth Graph &  & $1.8$ & $64.22$ & time \\
    Harries Graph &  & $15.0$ & $26.41$ & time \\
    Harries-Wong graph &  & $15.0$ & $26.41$ & time \\
    Horton Graph &  & $0.63$ & $26.54$ & time \\
    Klein 3-regular Graph &  & $10.59$ & $26.56$ & time \\
    Markstroem Graph &  & $3.1$ & $17.24$ & $8.25$ \\
    Meredith Graph &  & $2.95$ & $40.99$ & time \\
    Szekeres Snark Graph &  & $2.82$ & $26.54$ & time \\
        \end{tabular}
    \caption{Comparison of the bounds from Theorem \ref{theorem:lowerboundi(G^t)} (LP \eqref{eq:iso_LP_lower}) and Theorem \ref{theorem:upperboundi(G^t)} (LP \eqref{eq:iso_LP_upper}) for $t=4$ for Sage named graphs of diameter at least 5. Bounds in bold are tight.}
    \label{tab:t4_sim}
\end{table}

\newpage

\begin{table}[H]
    \centering
    \begin{tabular}{l|cccc}
    Graph & $\mu_2/2$ & CR$^\star$ bound & Bound~\eqref{eq:closed_lb_drg} & $i(G)$ \\
    \hline
    \textbf{Biggs-Smith graph} & $0.22$ & $0.22$ & \textbf{0.33} & time \\
    Coclique graph of Hoffmann-Singleton graph & $5.0$ & $5.0$ & $3.23$ & time \\
    Conway-Smith graph for 3S7 & $2.5$ & $2.54$ & $2.28$ & time \\
    \textbf{Coxeter Graph} & $0.5$ & $0.5$ & \textbf{0.56} & $0.69$ \\
    \textbf{Desargues Graph} & $0.5$ & $0.5$ & \textbf{0.6} & $0.6$ \\
    \textbf{Dodecahedron} & $0.38$ & $0.38$ & \textbf{0.6} & $0.6$ \\
    \textbf{Foster Graph} & $0.28$ & $0.28$ & \textbf{0.34} & time \\
    Foster graph for 3.Sym(6) graph & $1.5$ & $1.53$ & $1.38$ & time \\
    Gosset Graph & $9.0$ & $9.0$ & $9.0$ & time \\
    Heawood graph & $0.79$ & $0.79$ & $0.78$ & $1.0$ \\
    Hexahedron & $1.0$ & $1.0$ & $1.0$ & $1.0$ \\
    \textbf{Icosahedron} & $1.38$ & $1.38$ & \textbf{1.67} & $1.67$ \\
    Klein 7-regular Graph & $2.18$ & $2.18$ & $2.05$ & $2.5$ \\
    \textbf{Pappus Graph} & $0.63$ & $0.63$ & \textbf{0.66} & $0.78$ \\
    Perkel Graph & $1.69$ & $1.72$ & $1.36$ & time \\
    Sylvester Graph & $1.5$ & $1.5$ & $1.2$ & $1.67$ \\
    \textbf{Tutte 12-Cage} & $0.28$ & $0.28$ & \textbf{0.33} & time \\
    \textbf{Tutte-Coxeter graph} & $0.5$ & $0.5$ & \textbf{0.54} & $0.6$ \\
    Wells graph & $1.38$ & $1.38$ & $1.25$ & $1.5$ \\
    \end{tabular}
    \caption{Comparison of $\mu_2/2$, the CR$^\star$ lower bound from~\cite{DM2019}, and the bound~\eqref{eq:closed_lb_drg} for Sage named distance-regular graphs with diameter at least 3. Graphs for which the bound~\eqref{eq:closed_lb_drg} outperforms $\mu_2/2$ are highlighted. Values of $i(G)$ marked as `time' did not finish within 5 minutes.}
    \label{tab:t5_drg}
\end{table}

\end{document}